\documentclass[11pt,reqno,twoside]{amsart}

\usepackage{amsfonts,amsmath,amssymb}
\usepackage{mathrsfs,mathtools}
\usepackage{enumerate}
\usepackage{hyperref}
\usepackage{esint}
\usepackage{graphicx}
\usepackage{bm}
\usepackage{commath}
\usepackage{esint}
\DeclareMathAlphabet{\mathpzc}{OT1}{pzc}{m}{it}

\usepackage[textsize=small]{todonotes}
\usepackage{caption}

\numberwithin{equation}{section}

\usepackage[dvips,bottom=1.4in,right=1in,top=1in, left=1in]{geometry}

\usepackage[utf8]{inputenc}
\usepackage[english]{babel}
\usepackage{array}
\usepackage[capitalize]{cleveref}
\usepackage{enumitem}

\makeatletter
\def\eqnarray{\stepcounter{equation}\let\@currentlabel=\theequation
\global\@eqnswtrue
\tabskip\@centering\let\\=\@eqncr
$$\halign to \displaywidth\bgroup\hfil\global\@eqcnt\z@
  $\displaystyle\tabskip\z@{##}$&\global\@eqcnt\@ne
  \hfil$\displaystyle{{}##{}}$\hfil
  &\global\@eqcnt\tw@ $\displaystyle{##}$\hfil
  \tabskip\@centering&\llap{##}\tabskip\z@\cr}

\def\endeqnarray{\@@eqncr\egroup
      \global\advance\c@equation\m@ne$$\global\@ignoretrue}

\def\@yeqncr{\@ifnextchar [{\@xeqncr}{\@xeqncr[1pt]}}
\makeatother

\DeclareMathAlphabet\gothic{U}{euf}{m}{n}

\newcommand{\gota}{\gothic{a}}
\newcommand{\gotb}{\gothic{b}}
\newcommand{\gotm}{\gothic{m}}

\newtheorem{thm}{Theorem}[section]

\newtheorem{defn}[thm]{Definition}

\newtheorem{lem}[thm]{Lemma}

\newtheorem{proposition}[thm]{Proposition}

\newtheorem{remark}[thm]{Remark}

\numberwithin{equation}{section}

\def\Omc{\mathbb{R}^n\setminus\Omega}
\def\Omb{\mathbb{R}^n\setminus\overline{\Omega}}

\def\RR{{\mathbb{R}}}
\def\NN{{\mathbb{N}}}

\def\Om{\Omega}

\def\bOm{\overline{\Om}}
\def\pOm{\partial\Omega}

\title[The Laplace operator with nonlocal exterior conditions]{Realization of the fractional Laplacian with nonlocal exterior conditions via forms method}

\author{Burkhard Claus}
\address{B. Claus,Technische Universit\"at Dresden. Institut f\"ur Analysis
D-01062 Dresden (Germany)}
\email{burkhard.claus@tu-dresden.de}

\author{Mahamadi Warma}
\address{M. Warma, Department of Mathematical Sciences,  George Mason University. Fairfax, VA 22030 (USA). }
\email{mwarma@gmu.edu}

\thanks{The  second author is partially supported by the Air Force Office of Scientific Research under Award NO:  FA9550-18-1-0242}



\begin{document}

\begin{abstract}
Let $\Omega\subset\RR^n$ ($n\ge 1$) be a bounded open set with a Lipschitz continuous boundary.
In the first part of the paper, using the method of bilinear forms we give a characterization of the realization in $L^2(\Omega)$ of the fractional Laplace operator $(-\Delta)^s$ ($0<s<1$) with the nonlocal Neumann and Robin exterior conditions. Contrarily to the classical local case $s=1$,  it turns out that the nonlocal (Robin and Neumann) exterior conditions can be incorporated in the form domain.
We show that each of the above operators generates a strongly continuous submarkovian semigroup which is also ultracontractive. In the second part, we prove that the semigroup corresponding to the nonlocal Robin exterior condition is always sandwiched between the fractional Dirichlet semigroup and the fractional Neumann semigroup. 
\end{abstract}

\maketitle

\section{Introduction}

Let \(\Omega \subset \mathbb{R}^n\) ($n\ge 1$) be a bounded open set with a Lipschitz continuous boundary $\pOm$. The aim of the present paper is to give a  characterization of the realization in $L^2(\Omega)$ of the fractional Laplace operator $(-\Delta)^s$ ($0<s<1$) with the nonlocal Neumann and Robin exterior conditions by using the method of bilinear forms. Here the operator $(-\Delta)^s$ is given formally by the following singular integral:
\begin{align*}
(-\Delta)^su(x):=C_{n,s}\mbox{P.V.}\int_{\RR^n}\frac{u(x)-u(y)}{|x-y|^{n+2s}}\;dy,\;\;\;x\in\RR^n,
\end{align*}
where $C_{n,s}$ is a normalization constant depending on $n$ and $s$ only.
We refer to Section \ref{sec-prel} for a rigorous definition of $(-\Delta)^s$ and the class of functions for which the singular  integral exists. 

One of the main goals of the present article is to study elliptic problems of the form
\begin{align*}
    (-\Delta)^s u=f \text{ in }  \Om ,
\end{align*}
together with an appropriate exterior condition for \(u\) in \( \Omc\). That is, the Dirichlet \eqref{DiPr}, Neumann \eqref{NePr} or Robin problem \eqref{RoPr} for the fractional Laplacian on \(\Om\) as we shall precisely describe in Section 3. We want to derive a realization of the fractional Laplacian in \(L^2(\Om)\). Since \((-\Delta)^s\) is a nonlocal operator, one needs \(u\) to be defined on the whole \(\RR ^n\) in order to evaluate \((-\Delta)^s u\) at some \(x \in \Om\). Hence, to define a realization \(A: L^2(\Om) \supset D(A) \rightarrow L^2(\Om)\) of the fractional Laplacian, one needs to extend functions from \(\Om\) to the whole of \(\RR^n\). 
We show that the exterior conditions, that we shall define, correspond to certain extensions of functions from \(L^2(\Om)\) to the whole of \(\RR ^n\).  With the help of these extensions we shall define bilinear forms that yield realizations of the (weak) fractional Laplace operator. As in the case of the classical Laplace operator, each of these realizations $A$ is a positive and selfadjoint operator and thus, $-A$ generates a semigroup. These selfadjoint operators are the fractional version of the realizations of the classical Laplace operator with Dirichlet, Neumann and Robin boundary conditions.

It is nowadays well-known that the realizations in $L^2(\Om)$ of the Laplace operator ($-\Delta$) with the Neumann boundary conditions, $\partial_\nu u=0 $ on $\pOm$, and the Robin boundary conditions, $\partial_\nu u+\gamma u=0$ on $\pOm$,  are the selfadjoint operators on $L^2(\Om)$  associated with the closed bilinear forms
\begin{align}\label{lnbc}
\gota^N(u,v):=\int_{\Omega}\nabla u\cdot\nabla v\;dx, \;\; u, v\in D(\gota^N):=W^{1,2}(\Om),
\end{align}
and 
\begin{align}\label{lrbc}
\gota^R(u,v):=\int_{\Omega}\nabla u\cdot\nabla v\;dx+\int_{\pOm}\gamma uv\;d\sigma, \;\; u, v\in D(\gota^R):=W^{1,2}(\Om),
\end{align}
respectively. In \eqref{lrbc} $\gamma\in L^\infty(\pOm)$ is a non-negative given function. We refer to \cite{ArWa1,ArWa2,War-T} and the references therein for more details on this topic.

Another way to formulate boundary value problems for the fractional Laplacian on bounded domains is the regional fractional Laplacian $(-\Delta)_\Omega^s$ ($0<s<1$), defined formally by the following singular integral:
\begin{align*}
(-\Delta)_\Omega^su(x):=C_{n,s}\mbox{P.V.}\int_{\Omega}\frac{u(x)-u(y)}{|x-y|^{n+2s}}\;dy,\;\;\;x\in\Omega.
\end{align*}
This case is more similar to the classical local case $s=1$ and has been investigated in \cite{GW-CPDE,Guan,War-DN1,War} and their references. It turns out that for this case,  if $\frac 12<s<1$, then the associated normal derivative is a local operator (see e.g. \cite{Guan,War,War-IP}. If $0<s\le \frac 12$, then the Dirichlet and Neumann boundary conditions for $(-\Delta)_\Omega^s$ coincide (see e.g. \cite{Chen,War-DN1}). We mention that even on the space $\mathcal D(\Omega)$ of test functions, the operator $(-\Delta)_\Omega^s$ is different from $(-\Delta)^s$. More precisely, for $u\in \mathcal D(\Omega)$ we have 
\begin{align*}
(-\Delta)^su(x)=(-\Delta)_\Omega^su(x)+\kappa(x)u(x),\;x\in\Omega,\;\mbox{ where } \kappa(x):=C_{n,s}\int_{\Omc}\frac{dy}{|x-y|^{n+2s}}.
\end{align*} 
In this paper we will not deal with the regional fractional Laplacian; the situation is more delicate and challenging in the case of $(-\Delta)^s$. To be more precise we have the following difficulties.

\begin{itemize}
\item Firstly, let $u$ be a given function defined on $\Omega$. As we have already mentioned, in order to evaluate $(-\Delta)^su$ at a point $x\in \Om$, it is necessary to know $u$ in all $\RR^n$. Therefore, a natural question arises. Is it possible to find extensions \(u_D,u_N\) and \(u_R\) of \(u\) to the whole of \(\RR^n\) in the case of the Dirichlet, Neumann and Robin exterior conditions, respectively, such that these extensions solve the associated elliptic problem in \(\Omega\)?


\item Secondly, for elliptic problems associated with $(-\Delta)^s$ to be well-posed in $\Omega$, the conditions must not be prescribed on the boundary $\pOm$, but instead in $\Omc$. We shall call such a condition, an exterior condition. This shows that the condition must be given in terms of the extension $\tilde u$ instead of $u$.

\item Finally, it turns out that the operator playing the role for $(-\Delta)^s$ that the normal derivative does for $\Delta$ is also a nonlocal operator. Therefore, we have to deal with a double non-locality.
\end{itemize}

For functions $u,v\in W_\Omega^{s,2}$ (see  Section \ref{sec-prel} for the definition and more details of this space) we let
\begin{align*}
\mathcal E(u,v):=&\frac{C_{n,s}}{2}\int_{\Omega}\int_{\Omega}\frac{(u(x)-u(y))(v(x)-v(y))}{|x-y|^{n+2s}}\;dydx\\
&+C_{n,s}\int_{\Omega}\int_{\Omc}\frac{(u(x)-u(y))(v(x)-v(y))}{|x-y|^{n+2s}}\;dydx
\end{align*}
which is the bilinear form considered in \cite{SDipierro_XRosOton_EValdinoci_2017a}. Observe that
\begin{align*}
\mathcal E(u,v)=\frac{C_{n,s}}{2}\int\int_{\RR^{2n}\setminus(\RR^n\setminus\Omega)^2}\frac{(u(x)-u(y))(v(x)-v(y))}{|x-y|^{n+2s}}\;dydx.
\end{align*}
In particular, if $u=0$ in $\RR^n\setminus\Omega$ or $v=0$ in $\RR^n\setminus\Omega$, then
\begin{align*}
\mathcal E(u,v)=\frac{C_{n,s}}{2}\int_{\RR^n}\int_{\RR^n}\frac{(u(x)-u(y))(v(x)-v(y))}{|x-y|^{n+2s}}\;dydx.
\end{align*}

In the present paper we have obtained the following specific results.

\begin{enumerate}
\item[(i)] Firstly, for a function $u\in L^2(\Omega)$ we define its extension $u_N$ to $\RR^n$ as follows:
\begin{equation*} 
u_N(x):=
\begin{cases}
u(x)\;&\mbox{ if } x\in\Omega,\\
\displaystyle \frac{1}{\rho(x)}\int_{\Omega}\frac{u(y)}{|x-y|^{n+2s}}\;dy\;\;\;&\mbox{ if } x\in\Omb,
\end{cases}
\end{equation*}
where the function $\rho$ is given by
\begin{align*}
\rho(x):=\int_{\Om}\frac{1}{|x-y|^{n+2s}}\;dy,\;\;x\in\Omb.
\end{align*}
Our first main result (Theorem \ref{thm-38}) shows that the realization in $L^2(\Om)$ of $(-\Delta)^s$ with the nonlocal Neumann exterior condition is the selfadjoint operator $A_N$ associated with the closed, symmetric and densely defined bilinear form $\gota_N:D(\gota_N)\times D(\gota_N)\to\RR$ given by
\begin{align*} 
D(\gota_N):=\Big\{u\in L^2(\Om):\; u_N\in W_\Omega^{s,2} \Big\}\;\mbox{ and } \gota_N(u,v):=\mathcal E(u_N,v_N).
\end{align*}
The {\bf nonlocal Neumann exterior condition} is characterized by
\begin{align}\label{nbc}
\mathcal N^su_N=0\;\;\mbox{ in }\;\Omb,
\end{align}
where the operator $\mathcal N^s$ is defined for a function $v\in W_\Omega^{s,2}$ by
\begin{align*}
\mathcal N^sv(x):=C_{n,s}\int_{\Om}\frac{v(x)-v(y)}{|x-y|^{n+2s}}\;dy,\;\;x\in\Omb.
\end{align*}
We prove that  $-A_N$ generates a submarkovian semigroup $T_N$ on $L^2(\Omega)$ which is also ultracontractive in the sense that it maps $L^1(\Om)$ into $L^\infty(\Om)$.

\item[(ii)] Our second main result (Theorem \ref{thm-311}) concerns the nonlocal Robin exterior condition. For this, let $\beta\in L^1(\Omc)$ be a non-negative given function. 
For a function $u\in L^2(\Om)$ we define its extension \(u_R\) as follows:
\begin{align*} 
    u_R(x):=\begin{cases} u(x) &\text{ if } x \in \Omega, \\
 \displaystyle   \frac{C_{n,s}}{C_{n,s}\rho(x)+\beta(x)}\int_\Omega \frac{u(y)}{\vert x-y\vert^{n+2s}} dy &\text{ if } x\in\Omb.
    \end{cases}
\end{align*}
The realization in $L^2(\Om)$ of $(-\Delta)^s$ with the nonlocal Robin exterior condition is the selfadjoint operator $A_R$ associated with the closed, symmetric and densely defined bilinear form $\gota_R:D(\gota_R)\times D(\gota_R)\to\RR$ given by
\begin{align*} 
    D(\gota_R):=\Big\{ u \in L^2(\Om):\;u_R \in W^{s,2}_\Omega\cap L^2(\Omc,\beta dx)\Big \},
\end{align*}
 and 
$$\gota_R(u,v):=\mathcal E(u_R,v_R) +\int_{\Omc}\beta u_Rv_R\;dx.$$
The {\bf nonlocal Robin exterior condition} is characterized by
\begin{align}\label{rbc}
\mathcal N^su_R+\beta u_R=0\;\;\mbox{ in }\;\Omb.
\end{align}
We obtain that $-A_R$ generates a submarkovian semigroup $T_R$ on $L^2(\Omega)$ which is also ultracontractive.

\item[(iii)] Our third main result (Theorem \ref{thm-43}) shows that the semigroup $T_R$ is always sandwiched between the semigroup $T_D$ on $L^2(\Om)$ generated by the realization of $(-\Delta)^s$ in $L^2(\Om)$ with the zero {\bf Dirichlet exterior condition} $\tilde u=0$ in $\Omc$ and the semigroup $T_N$. That is, we have 
\begin{align*}
0\le T_D\le T_R\le T_N
\end{align*}
in the sense of \eqref{dom} below.
\end{enumerate}

Another novelty of the present paper is that contrarily to the local case $s=1$ or the regional fractional Laplace case, where the proofs of the submarkovian property and the domination of the semigroups are standard, for the case of $(-\Delta)^s$ investigated here, the proofs of the mentioned results require a careful analysis of the associated bilinear forms.

Let us mention that we shall give in Section \ref{sec-3} an alternative definition where the bilinear forms $\gota_N$ and $\gota_R$ are given by the infimum of certain functions, but this definition is more difficult to use to prove most of the results obtained in the present paper.

Fractional order operators (in particular the fractional Laplacian) have recently emerged as a modeling alternative in various branches of science.  They usually describe anomalous diffusion. A number of stochastic models for explaining anomalous diffusion have been introduced in the literature; among them we  quote the fractional Brownian motion; the continuous time random walk;  the L\'evy flights; the Schneider gray Brownian motion; and more generally, random walk models based on evolution equations of single and distributed fractional order in  space (see e.g. \cite{DS,GR,Man,Sch}).  In general a fractional diffusion operator corresponds to a diverging jump length variance in the random walk. In the literature the fractional Laplace operator is known as the generator of the so called $s$-stable L\'evy process.

The rest of the paper is structured as follows. In Section \ref{sec-prel} we introduce the function spaces needed to study our problem and recall some well-known results on Dirichlet forms and domination of semigroups that are used throughout the paper. In Section \ref{sec-3} we give a characterization of the realizations in $L^2(\Om)$ of $(-\Delta)^s$ with the three exterior conditions (Dirichlet, Neumann and Robin). We show that each of  these operators generates a  submarkovian semigroup which is also ultracontractive. The result concerning the domination of the semigroups is contained in Section \ref{sec-dom}. We conclude the paper by given some open problems in Section \ref{open}.

\section{Functional setup and preliminaries}\label{sec-prel}

Here we introduce the function spaces needed to investigate our problem, give a rigorous definition of $(-\Delta)^s$  and recall some known results on semigroups theory.

\subsection{Fractional order Sobolev spaces and the fractional Laplacian} Unless otherwise stated, $\Om \subset\RR^n$ ($n \ge 1$) is an arbitrary bounded  open set and $0 < s < 1$ is a real number.  Let 
 \begin{align}\label{SB}
    W^{s,2}(\Om) := \left\{ u \in L^2(\Om) :\; 
            \int_\Om\int_\Om \frac{|u(x)-u(y)|^2}{|x-y|^{n+2s}}\;dxdy < \infty \right\} 
 \end{align}
be endowed with the norm 
 \begin{align*}
    \|u\|_{W^{s,2}(\Om)} := \left(\int_\Om |u|^2\;dx 
        + \int_\Om\int_\Om  \frac{|u(x)-u(y)|^2}{|x-y|^{n+2s}}\;dxdy \right)^{\frac12}.
 \end{align*}
We define 
 \begin{align*}
    W^{s,2}_0(\overline\Om) := \left\{ u \in W^{s,2}(\RR^n) :\; u = 0 \mbox{ in } 
            \RR^n\setminus\Om \right\}=\Big\{u\in W^{s,2}(\RR^n):\; \mbox{supp}[u]\subset\bOm\Big\},
 \end{align*}
 where $W^{s,2}(\RR^n)$ is defined as in \eqref{SB} with $\Omega$ replaced by $\RR^n$. We have used $\bOm$ in the definition of $W^{s,2}_0(\overline\Om)$ in the spirit to avoid a confusion with the well-known space $W_0^{s,2}(\Omega):=\overline{\mathcal D(\Omega)}^{W^{s,2}(\Omega)}$.
 We set
 \begin{align*}
 \widetilde  W^{s,2}_0(\Om):=\Big\{u|_{\Omega}:u\in  W^{s,2}_0(\overline\Om)\Big\}.
 \end{align*}
We also define the  local fractional order Sobolev space
 \begin{equation}\label{eq:Ws2loc}
    W^{s,2}_{\rm loc}(\Omb) := \left\{ u \in L_{\rm loc}^2(\Omb):\; u\varphi \in W^{s,2}(\Omb) 
         \; \forall \ \varphi \in \mathcal{D}(\Omb) \right\}.  
 \end{equation} 
 
 \begin{remark}\label{rem-sob}
 {\em 
It is well-known that the following continuous embeddings hold:
\begin{align}\label{sob-em}
 \widetilde  W^{s,2}_0(\Om)\hookrightarrow
\begin{cases}
L^{\frac{2n}{n-2s}}(\Omega)\;\;&\mbox{ if }\; n>2s,\\
L^p(\Omega)\;\qquad \qquad\forall\;  p\in [1,\infty) \;&\mbox{ if }\; n=2s,\\
C^{0,1-\frac{n}{2s}}(\bOm)&\mbox{ if }\; n<2s.
\end{cases}
\end{align}
In addition to (\ref{sob-em}), we know that the embedding $\widetilde  W^{s,2}_0(\Om)\hookrightarrow L^2(\Omega)$ is compact.  If $\Omega$ has a Lipschitz continuous boundary, then \eqref{sob-em} also holds with $\widetilde  W^{s,2}_0(\Om)$ replaced with $W^{s,2}(\Omega)$. We refer to \cite[Chapter 1]{Gris}  for the proof of the above results (see also \cite{NPV} and the references therein).
}
\end{remark}

We have the following result.

\begin{lem}\label{lem-Girs}
The following assertions hold.
\begin{enumerate}
\item $\mathcal D(\Om)\subset \widetilde  W^{s,2}_0(\Om)$.
\item If $\Om$ has a continuous boundary, then $\mathcal D(\Omega)$ is dense in $\widetilde  W^{s,2}_0(\Om)$.
\end{enumerate}
\end{lem}

\begin{proof}
The proof  is contained in \cite[Theorem 1.4.2.2 and Corollary 1.4.4.5]{Gris} (see also \cite{Val}). 
\end{proof}

\begin{remark}\label{rem-obs}
{\em
We observe the following facts.
Let
 \begin{align*}
W_0^{s,2}(\Omega):=\overline{\mathcal D(\Omega)}^{W^{s,2}(\Omega)}.
\end{align*}
Assume that $\Omega$ has a Lipschitz continuous boundary $\pOm$.
Then, by \cite[Corollary 1.4.4.10]{Gris} for every $0<s<1$, 
\begin{align} \label{sobo2}
\widetilde  W^{s,2}_0(\Om)=\Big\{u\in W_0^{s,2}(\Omega):\; \frac{u}{\delta^s}\in L^2(\Omega)\Big\},
\end{align}
where $\delta(x):=\mbox{dist}(x,\pOm)$, $x\in\Omega$.
By \cite[Corollary 1.4.4.5]{Gris} if $s\ne \frac 12$, then $W_0^{s,2}(\Omega)=\widetilde  W^{s,2}_0(\Om)$. But if $s=\frac 12$, then $\widetilde  W^{\frac 12,2}_0(\Om)$ is a proper subspace of $W_0^{\frac 12,2}(\Omega)$. Notice also that $W_0^{s,2}(\Omega)=W^{s,2}(\Omega)$ for every $0<s\le \frac 12$ (see e.g. \cite{Chen,Gris,War}).
}
\end{remark}

For more information on fractional order Sobolev spaces we refer to \cite{NPV,Val,Gris,War}.

Next, let $\beta\in L^1(\Omc)$ be fixed and define the fractional order Sobolev type space
\begin{align*}
W_{\beta,\Omega}^{s,2}:=\Big\{u:\RR^n\to\RR\;\mbox{ measurable}:\;\|u\|_{W_{\beta,\Omega}^{s,2}}<\infty\Big\}
\end{align*}
where 
\begin{align}\label{norm-RV}
\|u\|_{W_{\beta,\Omega}^{s,2}}:=\left(\int_{\Omega}|u|^2\;dx+\int_{\Omc}|u|^2|\beta|\;dx+\int_{\RR^{2n}\setminus(\RR^n\setminus\Omega)^2}\frac{|u(x)-u(y)|^2}{|x-y|^{n+2s}}\;dxdy\right)^{\frac 12},
\end{align}
and
 \begin{align*}
\RR^{2n}\setminus(\RR^n\setminus\Om)^2
       = (\Om\times\Om)\cup(\Om\times(\RR^n\setminus\Om))\cup((\RR^n\setminus\Om)\times\Om).
\end{align*}
The space $W_{\beta,\Omega}^{s,2}$ has been introduced in \cite{SDipierro_XRosOton_EValdinoci_2017a} to study the Neumann problem for $(-\Delta)^s$ (see \eqref{NePr}). It also appears in a more general form in \cite{FKV15} and has been used there to study the Dirichlet problem for $(-\Delta)^s$ (see \eqref{DiPr}). If $\beta=0$, then we shall denote $W_{0,\Omega}^{s,2}=W_{\Omega}^{s,2}$. It is clear that $W_{\beta,\Omega}^{s,2}\hookrightarrow W_{\Omega}^{s,2}$.

The proof of the following result is contained in \cite[Proposition 3.1]{SDipierro_XRosOton_EValdinoci_2017a}.

\begin{lem}\label{lem22}
Let $\beta\in L^1(\Omc)$. Then $W_{\beta,\Omega}^{s,2}$ endowed with the norm \eqref{norm-RV} is a Hilbert space.
\end{lem}
 
To introduce the fractional Laplace operator  we set 

\begin{equation*}
\mathcal{L}_s^{1}(\RR^n):=\left\{u:\RR^n\rightarrow
\mathbb{R}\;\mbox{ measurable}: \;\int_{\RR^n}\frac{|u(x)|}{(1+|x|)^{n+2s}}%
\;dx<\infty \right\}.
\end{equation*}
For $u\in \mathcal{L}_s^{1}(\RR^n)$ and $ 
\varepsilon >0$ we let
\begin{equation*}
(-\Delta )_{\varepsilon }^{s}u(x):=C_{n,s}\int_{\left\{y\in \RR^n:|y-x|>\varepsilon \right\}}
\frac{u(x)-u(y)}{|x-y|^{n+2s}}dy,\;\;x\in\RR^n,
\end{equation*}%
where the normalization constant $C_{n,s}$ is given by
\begin{equation}\label{CN}
C_{n,s}:=\frac{s2^{2s}\Gamma\left(\frac{2s+n}{2}\right)}{\pi^{\frac
n2}\Gamma(1-s)},
\end{equation}%
and $\Gamma $ is the usual Euler Gamma function. The {\bf fractional Laplace operator} 
$(-\Delta )^{s}$ is defined for $u\in \mathcal{L}_s^{1}(\RR^n)$  by the formula:
\begin{align}
(-\Delta )^{s}u(x):=C_{n,s}\mbox{P.V.}\int_{\RR^n}\frac{u(x)-u(y)}{|x-y|^{n+2s}}dy 
=\lim_{\varepsilon \downarrow 0}(-\Delta )_{\varepsilon
}^{s}u(x),\;\;x\in\RR^n,\label{eq11}
\end{align}%
provided that the limit exists for a.e. $x\in\RR^n$. We have that $\mathcal{L}_s^{1}(\RR^n)$ is the right space for which $v:=(-\Delta )_{\varepsilon }^{s}u$ exists for every $\varepsilon>0$ and $v$ being also continuous at the continuity points of $u$.  Throughout the following we shall write \((-\Delta)^s u \in L^2(\Om)\) if the limit in \eqref{eq11} exists almost everywhere, and the function \(x \mapsto (-\Delta)^s u(x)$ belongs to $L^2(\Om)\). 

 If, for a given function \(u\), \((-\Delta )^{s}u \in L^2(\RR^n)\), then 
\begin{align*}
(-\Delta)^su=\frac{1}{\Gamma(-s)}\int_0^\infty\left(e^{t\Delta}u-u\right)\frac{dt}{t^{1+s}}
\end{align*}
where $\Gamma(-s)=-\frac{\Gamma(1-s)}{s}$ and $(e^{t\Delta})_{t\ge 0}$ is the semigroup on $L^2(\mathbb R^n)$ generated by  $\Delta$.  That is, we can define $(-\Delta)^s$ as the fractional $s$-power of the classical Laplacian \(-\Delta\).  We refer to \cite{S-T} and their references for more details.
Furthermore, $(-\Delta )^{s}$ can be also defined as the pseudo-differential operator with symbol $|\xi|^{2s}$ by using Fourier transforms (see e.g. \cite{NPV}).

For more details on the fractional Laplace operator, we refer to  \cite{BCF,Caf3,Caf1,Caf2,NPV,War} and their references.

Next, for $u \in W_\Omega^{s,2}$ we define the nonlocal normal derivative $\mathcal{N}^su$ of $u$ as follows:
 \begin{align}\label{NLND}
    \mathcal{N}^s u(x) := C_{n,s} \int_\Om \frac{u(x)-u(y)}{|x-y|^{n+2s}}\;dy, 
            \quad x \in \RR^n \setminus \overline\Om . 
 \end{align}
Clearly, $\mathcal{N}^s$ is a nonlocal operator and is well defined on $ W_\Omega^{s,2}$ as shows the following result.

 \begin{lem}\label{lem:Nmap}
  The nonlocal normal derivative $\mathcal{N}^s$ maps $W_\Omega^{s,2}$ continuously into
 $ L^{2}_{\rm loc}(\Omb)$. 
 \end{lem}
 
 \begin{proof}
 It has been shown in \cite[Lemma~3.2]{ghosh2016calder} that $\mathcal{N}^s$ maps $ W_\Omega^{s,2}$ into
 $ W^{s,2}_{\rm loc}(\Omb)\subset  L^{2}_{\rm loc}(\Omb)$.  The continuity of the mapping can be shown by using the formula of $\mathcal{N}^su$ given in \cite[Lemma~3.2]{ghosh2016calder}.
 \end{proof}
 
Despite the fact that $\mathcal{N}_s$ is defined on $\Omb$, it
is still known as a ``normal derivative''. This is due to its similarity with the 
classical normal derivative as shows the following result.

 \begin{proposition}\label{prop:prop}
 Let $\Omega\subset\RR^n$ be a bounded open set with a Lipschitz continuous boundary.
 Then the following assertions hold.
  \begin{enumerate}
    \item {\bf The divergence theorem}:  Let $u\in C_0^2(\RR^n):=\left\{u\in C^2(\RR^n):\; \lim_{|x|\to\infty} u(x)=0\right\}$. Then
  \begin{align*}
       \int_\Om (-\Delta)^s u\;dx = -\int_{\RR^n\setminus\Om} \mathcal{N}^s u\;dx.
     \end{align*}
      
    \item {\bf The integration by parts formula}:   Let $u \in W_\Omega^{s,2}$ be such that $(-\Delta)^su \in L^2(\Omega)$ and $\mathcal N_su\in L^2(\Omc)$. Then for every $v\in W_\Omega^{s,2}\cap L^2(\Omc)$ we have 
    \begin{align}\label{Int-Part}
       \int_\Om v(-\Delta)^s u\;dx =& \frac{C_{n,s}}{2} 
        \int\int_{\RR^{2n}\setminus(\RR^n\setminus\Om)^2} 
         \frac{(u(x)-u(y))(v(x)-v(y))}{|x-y|^{n+2s}} \;dxdy 
         - \int_{\RR^n\setminus\Om} v\mathcal{N}^s u\;dx  \nonumber \\
         &= \mathcal{E} (u,v) -  \int_{\RR^n\setminus\Om} v\mathcal{N}^s u\;dx
      \end{align}

    \item {\bf The limit as $s\uparrow 1^-$}:    Let $u,v\in C_0^2(\RR^n)$. Then
      \begin{align*}
       \lim_{s\uparrow 1^-}\int_{\RR^n\setminus\Om} v \mathcal{N}^s u\;dx 
        = \int_{\pOm} v\partial_{\nu}u\; d\sigma . 
      \end{align*}  
  \end{enumerate}
 \end{proposition}
 
  \begin{proof}
 The proofs of (a) and (c) are contained in \cite[Lemma 3.2]{SDipierro_XRosOton_EValdinoci_2017a} and \cite[Proposition 5.1]{SDipierro_XRosOton_EValdinoci_2017a}, respectively. The proof of (b) for smooth functions can be found in \cite[Lemma 3.3]{SDipierro_XRosOton_EValdinoci_2017a}. The version given here is obtained by using a density argument (see e.g. \cite[Proposition~3.7]{warma2018approximate}).
 \end{proof}
 
 \begin{remark} 
  \rm{
  Comparing the properties (a)-(c) in Proposition~\ref{prop:prop} with the properties of the Laplacian $\Delta$, we can immediately deduce that $\mathcal{N}^s$ plays the same role for $(-\Delta)^s$ as the classical normal derivative $\partial_\nu$ plays for $-\Delta$  and \(\mathcal{E}\) takes the role of the classical Dirichlet integral 
  \begin{align*}
      (u,v) \mapsto \int_\Om \nabla u \cdot\nabla v\;dx.
  \end{align*} 
  For this reason we call $\mathcal{N}^s$ the {\bf nonlocal normal derivative}. The name {\bf interaction operator} has also been used for $\mathcal N^s$ in \cite{AKW-IP,DGLZ}.
  }
 \end{remark}

\subsection{Dirichlet forms and domination of semigroups}

Let $X$ be a locally compact separable metric space. Let $\gotm$ be a Radon measure on $X$ and assume that $\operatorname{supp}[\gotm]=X$. 

We recall the following notion of energy forms, cf. \cite[Chapter 1]{Fuk} (see also \cite[Chapter 1]{Dav}).

\begin{defn}
\label{Diri-form}The form $(\gota,D(\gota))$ is said to be a Dirichlet form
if the following conditions hold:

\begin{enumerate}
\item $\gota:D(\gota)\times D(\gota)\rightarrow \mathbb{R}$ where
 $D(\gota)$  is a dense
linear subspace of $L^{2}(X):=L^2(X,\gotm)$.

\item $\gota\left( \cdot,\cdot\right)$ is a symmetric and
non-negative bilinear form. 

\item Let $\lambda >0$ and define $\gota_{\lambda }\left( u,v\right)
:=\gota\left( u,v\right) +\mathcal{\lambda }\left( u,v\right)_{L^2(X)} $
for $u,v\in D(\gota)$. The form $\gota$ is said to be closed, if $(u_{k})_{k\in\NN}\subset D(\gota)$ with
\begin{equation*}
\gota_{\lambda }\left( u_{k}-u_{m},u_{k}-u_{m}\right) \rightarrow 0%
\text{ as }k,m\rightarrow \infty ,
\end{equation*}%
then there exists $u\in D(\gota)$ such that
\begin{equation*}
\gota_{\lambda }\left( u_{k}-u,u_{k}-u\right) \rightarrow 0\text{ as
}k\rightarrow \infty .
\end{equation*}

\item $u\in D(\gota)$ implies $u\wedge 1\in D(\gota)$ and
\begin{align*}
\gota(u\wedge 1,(u-1)^+)\ge 0
\end{align*}
where $u\wedge v:=\min(u,v)$ and $v^+:=\max(v,0)$ for $u,v\in L^2(X)$.
\end{enumerate}
\end{defn}

Definition \ref{Diri-form}(d), taken from \cite[Proposition 6.7]{AR-TOM}, is equivalent to \cite[Formula (1.1.6)]{Fuk}.

There is a one-to-one correspondence between the
family of closed, symmetric, densely defined bilinear forms $(\gota, D(\gota))$ on $L^{2}\left( X\right)$
and the family of \emph{non-negative} (definite) \emph{selfadjoint}
operators $A$ on $L^{2}\left( X\right) $ defined by
\begin{equation}\label{op-A}
\begin{cases}
D(A):=\Big\{u\in D(\gota):\;\exists\;f\in L^2(X)\mbox{ such that } \gota(u,v)=(f,v)_{L^2(X)}\;\forall\;v\in D(\gota)\Big\},\\
Au:=f.
\end{cases}
\end{equation}
In that case the operator $-A$ generates a strongly continuous semigroup $(e^{-tA})_{t\ge 0}$ on $L^2(X)$.

Throughout the following if $1\le p\le\infty$, then we shall let $L^p(X):=L^p(X,\gotm)$.

\begin{remark}\label{rem1}
{\em 
Assume that the form $(\gota,D(\gota))$  satisfies the conditions (a)-(c) in Definition \ref{Diri-form}.  Then  Definition \ref{Diri-form}(d) can be replaced by the following two conditions:

\begin{enumerate}
\item[(i)] $u\in D(\gota)$ implies $|u|\in D(\gota)$ and $\gota(|u|,|u|)\le \gota(u,u)$. In that case $(e^{-tA})_{t\ge 0}$ is said to be {\bf positivity-preserving} in the sense that $u\in L^2(X)$ and $u\ge 0$ implies $e^{-tA}u\ge 0$.

\item[(ii)]  $0\le u\in D(\gota)$ implies $u\wedge 1\in D(\gota)$ and $\gota(u\wedge 1,u\wedge 1)\le\gota(u,u)$. In that case  $(e^{-tA})_{t\ge 0}$ is said to be {\bf $L^\infty$-contractive} in the sense that for every $t\ge 0$ and $u\in L^2(X)\cap L^\infty(X)$,
\begin{align*}
\|e^{-tA}u\|_{L^\infty(X)}\le \|u\|_{L^\infty(X)}.
\end{align*}
\end{enumerate}
A positivity-preserving and $L^\infty$-contractive semigroup is called {\bf submarkovian}.
}
\end{remark}

Any selfadjoint operator $A$ that is in one-to-one correspondence
with a Dirichlet form $\left( \gota,D(\gota)\right)$ turns
out to possess a number of good properties provided a certain Sobolev
embedding theorem holds for $D(\gota)$ (see e.g. \cite{Dav,Fuk,Ouh-book}).

\begin{thm}\label{thm-main}
Let $A$ be the selfadjoint operator on $L^2(X)$ associated with a Dirichlet space $( \gota,D(\gota))$ in the sense of  \eqref{op-A} and $(e^{-tA})_{t\ge 0}$ the submarkovian semigroup on $L^2(X)$ generated by $-A$. 
Then the following assertions hold.

\begin{enumerate}
\item The semigroup $(e^{-tA})_{t\geq
0}$  can be extended to a
contraction semigroup on $L^{p}\left( X\right) $ for every $p\in \lbrack
1,\infty ]$. Each semigroup is strongly continuous if $p\in \lbrack
1,\infty )$ and bounded analytic if $p\in (1,\infty )$.

\item If in addition the continuous embedding
\begin{equation}
D(\gota)\hookrightarrow L^{2q_{\gota}}\left( X\right) \text{ for some }q_{\gota}>1,
\label{Sobolev}
\end{equation}%
holds, then the semigroup $(e^{-tA})_{t\geq 0}$ is ultracontractive in the sense that it maps $L^1(X)$ into $L^\infty(X)$. More precisely, there is a constant $C>0$ such that for every $f\in L^1(X)\cap L^2(X)$ we have 
\begin{align}\label{ultra}
\|e^{-tA}f\|_{L^\infty(X)}\le Ct^{-\frac{q_\gota}{q_\gota-1}}\|f\|_{L^1(X)}\;\mbox{ for all }\; 0<t\le 1.
\end{align}
\end{enumerate}
\end{thm}

\begin{remark}\label{rem-210}
{\em Let $(\gota,D(\gota))$, $A$ and $(e^{-tA})_{t\ge 0}$ be as in Theorem \ref{thm-main}.
Assume that $\gotm(X)<\infty$ and $D(\gota)\overset{c}{\hookrightarrow }L^{2}\left( X\right)$ (compact embedding). Then the operator $A$ has a compact resolvent. Hence, it has a discrete
spectrum which  is a non-decreasing sequence of real numbers, $
0\leq \lambda _{1}\leq \lambda _{2}\leq \cdots \leq \lambda _{n}\leq \cdots,$ satisfying $\lim_{n\to\infty}\lambda_n=+\infty$. In addition each semigroup on $L^{p}\left( X\right) $ is compact for every $p\in \lbrack 1,\infty ]$. If the embedding \eqref{Sobolev} also holds, then we have the following.
\begin{enumerate}
\item If $\lambda_1>0$, then the estimate \eqref{ultra} holds for every $t>0$.
\item If $\lambda_1=0$, then \eqref{ultra} can be replaced by
\begin{align*}
\|e^{-tA}f\|_{L^\infty(X)}\le Ce^tt^{-\frac{q_\gota}{q_\gota-1}}\|f\|_{L^1(X)}\;\mbox{ for all }\; t>0.
\end{align*}
\end{enumerate}}
\end{remark}

Next, let $S$ and $T$ be two semigroups on $L^2(X)$ and assume that $T$ is positivity-preserving. We shall say that $S$ is dominated by $T$, and we write $S\le T$, if
\begin{align}\label{dom}
|S(t)f|\le T(t)|f|\;\mbox{ for all }\; t\ge 0 \mbox{ and } f\in L^2(X).
\end{align}

The following domination criterion of semigroups has been obtained in \cite{Ouh-Dom}.

\begin{thm}\label{thm-29}
Let $S$ and $T$ be two symmetric semigroups on $L^2(X)$. Let $(\gota,D(\gota))$ and $(\gotb,D(\gotb))$ be the bilinear, symmetric and closed forms associated with $S$ and $T$, respectively. Assume that both semigroups are positivity-preserving. Then the following assertions are equivalent.
\begin{enumerate}
\item[(i)] The semigroup $S$ is dominated by the semigroup $T$ in the sense of \eqref{dom}.

\item[(ii)]\begin{itemize}
\item  $D(\gota)\subset D(\gotb)$ and if $0\le v\le u$ with $u\in D(\gota)$ and $v\in D(\gotb)$, then $v\in D(\gota)$. That is, $D(\gota)$ is an ideal in $D(\gotb)$.
\item For all $0\le u,v\in D(\gota)$ we have $\gotb(u,v)\le \gota(u,v)$.
\end{itemize}
\end{enumerate}
\end{thm}

For more information on domination criteria of semigroups we refer to \cite[Chapter 2]{Ouh-book}.

\section{The three exterior conditions for the fractional Laplacian}\label{sec-3}

In this section we introduce the realization in $L^2(\Om)$ of the fractional Laplace operator with the Dirichlet, the nonlocal Neumann and the nonlocal Robin exterior conditions. We will also give several qualitative properties of these operators.  

Throughout the remainder of the paper, for functions $u,v\in W_\Omega^{s,2}$ we shall let
\begin{align}\label{form-EN}
    \mathcal{E}(u,v):=& \frac{C_{n,s}}{2}\int_\Omega \int_\Omega \frac{(u(x)-u(y))(v(x)-v(y))}{\vert x-y\vert^{n+2s} } dy \;dx\notag\\ 
       & +C_{n,s}\int_\Omega \int_{\Omc} \frac{(u(x)-u(y))(v(x)-v(y))}{\vert x-y\vert^{n+2s} } dy \;dx
\end{align}
where $C_{n,s}$ is the constant given in \eqref{CN}.

\begin{remark}
{\em
We want to study the fractional Laplace operator by using  form methods. As in the case of the classical Laplace operator, form methods yield an existence theory for weak solutions. In the case of $(-\Delta)^s$, it is a priori not clear if weak solutions are in fact strong solutions. That is, it is not obvious if $(-\Delta)^su$ is well defined almost everywhere, lies in an appropriate function space, and the equation and exterior conditions are satisfied almost everywhere. }

{\em
In the classical case $s=1$ this is well known. More precisely, weak solutions belong to \(W^{2,2}_{\text{loc}}(\Omega)\) and the associated equation holds pointwise almost everywhere in $\Omega$. To the best of our knowledge similar results for the fractional Laplacian on bounded domains as defined in this paper,  are currently unknown. We shall come back on these issues after introducing our notion of weak solutions for each exterior condition.}

{\em 
Nevertheless, as in the classical case, strong solutions are always weak solutions for every exterior condition we investigated in this article. This follows directly from the integration by parts formula \eqref{Int-Part}  for the Neumann and Robin exterior conditions, and to \eqref{IN-Dir} below for the Dirichlet exterior condition. Hence, the operators given in this section are always selfadjoint realizations of the fractional Laplace operator.}
\end{remark}

\subsection{The Dirichlet exterior condition} \label{sub-dir}

Throughout this subsection $\Omega\subset\RR^n$ is an arbitrary bounded open set. Let $u\in W_0^{s,2}(\bOm)$ be such that $(-\Delta)^su\in L^2(\Omega)$. Then the following integration by parts formula is well-known (see e.g. \cite{NPV}). For every $v\in W_0^{s,2}(\bOm)$ we have 
\begin{align}\label{IN-Dir}
\int_{\Om}v(-\Delta)^su\;dx= \frac{C_{n,s}}{2}\int_{\RR^n} \int_{\RR^n} \frac{(u(x)-u(y))(v(x)-v(y))}{\vert x-y\vert^{n+2s} } dy \;dx=\mathcal E(u,v).
\end{align}

Several authors (see e.g. \cite{Grubb3, Grub,RS-DP2,RS-DP,warma2018approximate}) have studied the Dirichlet problem for $(-\Delta)^s$, that is, the elliptic problem
\begin{equation}\label{DiPr}
(-\Delta)^su=f\;\;\mbox{ in }\;\Omega,\;\;\; u=0\;\mbox{ in }\;\Omc,
\end{equation}
and the associated parabolic problem, but not in the same spirit as in the present paper. Even if this case is straightforward, for the sake of completeness and since we would like to make a comparison with the Neumann and Robin cases, we have decided to include it here.

Let $f\in L^2(\Omega)$. We shall say that a function $u\in W_0^{s,2}(\bOm)$ is a weak solution of \eqref{DiPr}, if 
\begin{align}\label{DiPrWe}
\mathcal E(u,v)=\frac{C_{n,s}}{2}\int_{\RR^n} \int_{\RR^n} \frac{(u(x)-u(y))(v(x)-v(y))}{\vert x-y\vert^{n+2s} } dy \;dx=\int_{\Omega}fv\;dx,
\end{align}
for every $v\in W_0^{s,2}(\bOm)$.

Using the classical Lax-Milgram lemma, it is straightforward to show the existence and uniqueness of weak solutions to the Dirichlet problem \eqref{DiPr}.

\begin{remark}\label{rem-Diri}
{\em We make the following observation.
Let $f\in L^2(\Omega)$ and $u\in W_0^{s,2}(\bOm)$ a weak solution of \eqref{DiPr}. 
We do not know if $u$ is a strong solution of \eqref{DiPr} in the sense that $(-\Delta)^su$ (as given in \eqref{eq11}) is well defined almost everywhere, $(-\Delta)^su\in L^2(\Omega)$ and $(-\Delta)^su=f$ a.e. in $\Omega$.
However, we have the following inner regularity properties of solutions to \eqref{DiPr}. 
By \cite[Theorem 1.3]{BWZ1} weak solutions of \eqref{DiPr} belong to $W_{\rm loc}^{2s,2}(\Omega)$. This maximal inner regularity result can be a starting point to investigate if weak and strong solutions coincide.
}
\end{remark}

For a function \(u \in L^2(\Omega)\) we define its extension $u_D$ as follows:
\begin{align}\label{uD}
    u_D(x):=\begin{cases} u(x) &\text{ if } x\in \Omega, \\
    0 &\text{ if } x\in\RR^n\setminus\Omega.
    \end{cases}
\end{align}

The following result characterizes the realization in $L^2(\Om)$ of  $(-\Delta)^s$ with the zero Dirichlet exterior condition via the method of bilinear forms.

\begin{thm}\label{thm-31}
Let 
\begin{align}\label{dom-DC}
    D(\gota_D):=\Big\{u \in L^2(\Omega): u_D\in W_\Om^{s,2}\Big\}=\widetilde W_0^{s,2}(\Om),
\end{align}
and \(\gota_D:D(\gota_D) \times D(\gota_D) \rightarrow \mathbb{R}\) the form given by
\begin{align*}
    \gota_D(u,v):=\mathcal{E}(u_D,v_D).
\end{align*}
Then $\gota_D$ is a densely defined, symmetric and closed bilinear form in $L^2(\Omega)$. The selfadjoint operator $A_D$ on $L^2(\Om)$ associated with \(\gota_D\)  in the sense of \eqref{op-A} is given by  
\begin{equation}\label{Op-Dir}
\begin{cases}
 D(A_D):=\Big\{u \in \widetilde W_0^{s,2}(\Om):\;\exists\; f\in L^2(\Om)\mbox{ such that } u_D \text{ is a weak solution of }  \eqref{DiPr}\\
 \hfill \text{ with right hand side } f \Big\},\\
 A_Du:=f.
 \end{cases}
\end{equation}
\end{thm}

\begin{proof}
Firstly, we notice that  $D(\gota_D)$ endowed with the norm $\|u\|_{D(\gota_D)}:=\|u_D\|_{W_0^{s,2}(\bOm)}$
 is a Hilbert space. Hence, the form $\gota_D$ is closed. Secondly, since $\mathcal D(\Omega)\subset\widetilde W^{s,2}_0(\Om)$ (by Lemma \ref{lem-Girs}(a)), we have that $\gota_D$ is densely defined.

Finally, by definition, \(u \in \widetilde W_0^{s,2}(\Om)\) if and only if \(u_D \in W_\Om^{s,2}\). Hence, the characterization of the operator $A_D$ given in \eqref{Op-Dir} is trivial. The proof is finished.
\end{proof}

\begin{remark}
{\em The operator $A_D$ is the realization in $L^2(\Omega)$ of $(-\Delta)^s$ with  the zero Dirichlet exterior condition. 
}
\end{remark}

Denote by $T_D=(e^{-tA_D})_{t\ge 0}$ the semigroup on $L^2(\Omega)$ generated by $-A_D$.

\begin{thm}\label{thm-33} 
The semigroup $T_D$ is positivity-preserving.
\end{thm}

\begin{proof}
Notice that for $u,v\in D(\gota_D)$ we have
\begin{align*}
\gota_D(u,v)=\mathcal E(u_D,v_D)=\frac{C_{n,s}}{2}\int_{\RR^n}\int_{\RR^n}\frac{(u_D(x)-u_D(y))(v_D(x)-v_D(y))}{|x-y|^{n+2s}}\;dx\;dy.
\end{align*}
Let \(u\in D(\gota_D)\). By Remark \eqref{rem1}(i)
we have to show that \(\vert u\vert\in D(\gota_D)\) and $\gota_D(|u|,|u|)\le \gota_D(u,u)$. Indeed, using the reverse triangle inequality we get that
\begin{align*}
    \gota_D(\vert u \vert,\vert u \vert)&=\frac{C_{n,s}}{2} \int_{\mathbb{R}^n} \int_{\mathbb{R}^n} \frac{\big\vert \vert u_D \vert (x) - \vert u_D \vert (y) \big\vert^2}{\vert x-y \vert ^{n+2s} }dx\;dy \\
      & \leq \frac{C_{n,s}}{2}\int_{\mathbb{R}^n} \int_{\mathbb{R}^n} \frac{\big\vert  u_D (x) - u_D (y) \big\vert^2}{\vert x-y \vert ^{n+2s} }dx\;dy \\
     &= \gota_D(u,u).
\end{align*}
Hence,  \(\vert u\vert\in D(\gota_D)\) and $  \gota_D(\vert u \vert,\vert u \vert)\le \gota_D(u,u)$.
The proof is finished.
\end{proof}

Other qualitative properties of the semigroup $T_D$ will be given in Section \ref{sec-dom}.

\begin{remark}\label{rem-D-form}
{\em We notice the following facts.
\begin{enumerate}
\item Let
\begin{align}\label{kappa}
\kappa(x):=C_{n,s}\int_{\Omc}\frac{1}{|x-y|^{n+2s}}\;dy,\;\;x\in\Om.
\end{align}
Then  a simple calculation shows that for $u,v\in W_0^{s,2}(\bOm)$ we have
\begin{align}\label{dota-a}
\mathcal E(u,v)=\frac{C_{n,s}}{2}\int_{\Om}\int_{\Om}\frac{(u(x)-u(y))(v(x)-v(y))}{|x-y|^{n+2s}}\;dx\;dy +\int_{\Om}uv\kappa\;dx.
\end{align}
It follows from \eqref{dota-a} that for $u,v\in D(\gota_D)=\widetilde W_0^{s,2}(\Om)$ we have
\begin{align}\label{dota}
\gota_D(u,v)=\mathcal E(u_D,v_D)=\frac{C_{n,s}}{2}\int_{\Om}\int_{\Om}\frac{(u(x)-u(y))(v(x)-v(y))}{|x-y|^{n+2s}}\;dx\;dy +\int_{\Om}uv\kappa\;dx,
\end{align}
where $\kappa$ is given in \eqref{kappa}.
\item Now assume that $\Omega$ has a Lipschitz continuous boundary $\pOm$. It has been shown in \cite[Formula (1.3.2.12)]{Gris} that there are two constants $0<C_1\le C_2$ such that
\begin{align}\label{delta}
C_1\delta(x)^{-2s}\le\kappa(x)\le C_2\delta(x)^{-2s},
\end{align}
where we recall that $\delta(x)=\mbox{dist}(x,\pOm)$, $x\in\Omega$. Taking into account the characterization of $\widetilde W_0^{s,2}(\Om)$ given in \eqref{sobo2} and the estimate \eqref{delta}, we can deduce that $\int_{\Om}|u(x)|^2\kappa(x)\;dx<\infty$ for every $u\in \widetilde W_0^{s,2}(\Om)$. If $s=\frac 12$, then in view of Remark \ref{rem-obs} we cannot guarantee that  $\int_{\Om}|u(x)|^2\kappa(x)\;dx<\infty$ for every $u\in W_0^{\frac 12,2}(\Omega)$. For this reason we cannot define the form $\gota_D$ in \eqref{dota} with the space $W_0^{s,2}(\Omega)$ at least for $s=\frac 12$.
\end{enumerate}
}
\end{remark}

\subsection{The nonlocal Neumann exterior condition} 
Throughout the remainder of the paper, $\Omega\subset\RR^n$ ($n\ge 1$) is a bounded open set with a Lipschitz continuous boundary. In \cite{SDipierro_XRosOton_EValdinoci_2017a} the authors have studied the well-posedness of the following elliptic Neumann problem:
\begin{equation}\label{NePr}
(-\Delta)^su=f\;\;\mbox{ in }\;\Omega,\;\;\; \mathcal N^su=0\;\mbox{ in }\;\Omb,
\end{equation}
and the associated parabolic problem. The eigenvalues problem associated to \eqref{NePr}  has been investigated  without describing explicitly the associated operator. We emphasize that the non-described operator in \cite{SDipierro_XRosOton_EValdinoci_2017a} is the one that we shall completely characterize in the present subsection. Problem \eqref{NePr} in another spirit has been also studied  in \cite{AKW-IP}.

For $f\in L^2(\Omega)$ we shall say that a function $u\in W_\Omega^{s,2}$ is a weak solution of \eqref{NePr} if 
\begin{align}\label{SN}
\mathcal E(u,v)=\frac{C_{n,s}}{2}\int\int_{\RR^{2n}\setminus(\RR^n\setminus\Omega)^2} \frac{(u(x)-u(y))(v(x)-v(y))}{\vert x-y\vert^{n+2s} } dy \;dx=\int_{\Omega}fv\;dx
\end{align}
for every $v\in W_\Omega^{s,2}$. Using the Lax-Milgram lemma, we can easily show that \eqref{NePr} has a weak solution.

Let $f\in L^2(\Omega)$ and $u\in W_\Omega^{s,2}$ a weak solution of \eqref{NePr}. We want to emphasize the following points.
\begin{itemize}
\item We do not know if $((-\Delta)^su)|_{\Omega}=f$ strongly, in the sense that (\(-\Delta)^su\) exists almost everywhere and $(-\Delta)^su=f$ a.e. in $\Omega$.
\item We do know that $\mathcal N^su=0$ a.e.  in $\Omb$. Indeed, taking $v\in\mathcal D(\Omb)$ as a test function in \eqref{SN} and calculating, we get that (notice that $v=0$ in $\Omega$)
\begin{align*}
0=\mathcal E(u,v)=&C_{n,s}\int_{\Omega}\int_{\RR^n\setminus\Omega} \frac{(u(x)-u(y))(v(x)-v(y))}{\vert x-y\vert^{n+2s} } dydx\\
=&C_{n,s}\int_{\RR^n\setminus\Omega}v(y)\left(\int_{\Omega} \frac{u(y)-u(x)}{\vert x-y\vert^{n+2s} } dx\right)dy\\
=&\int_{\RR^n\setminus\Omega}v(y)\mathcal N^su(y)dy.
\end{align*}
Since $v\in\mathcal D(\Omb)$ was arbitrary, it follows that  $\mathcal N^su=0$ a.e.  in $\Omb$.
\end{itemize}

Next, we recall that for a function $u\in W_\Omega^{s,2}$ we have defined $\mathcal N^su$ as follows:
\begin{align*}
    \mathcal{N}^s u (x):= C_{n,s} \int_\Omega \frac{u(x)-u(y)}{\vert x-y\vert ^{n+2s}} dy,\;\;x\in\Omb.
\end{align*}

Let
\begin{align}\label{rho}
    \rho(x):=\int_\Omega \frac{1}{\vert x-y\vert^{n+2s}} dy,\;\;x\in\Omb.
\end{align}

We have the following result.

\begin{lem}\label{Chara_Neumann_boundary}
Let \(u\in W^{s,2}_\Omega\). Then,
\begin{align*}
    \mathcal{N}^s u (x)=0,\;\;x\in\Omb,
\end{align*}
if and only if
\begin{align*}
    u(x)= \frac{1}{\rho(x)} \int_\Omega \frac{u(y)}{\vert x-y\vert ^{n+2s}}\;dy,\;\;x\in\Omb.
\end{align*}
\end{lem}

\begin{proof}
Let \(u \in W^{s,2}_\Omega\). Then, by definition the identity
\begin{align*}
    \mathcal{N}^s u (x)=0 \text{ for } x \in \Omb,
\end{align*}
is equivalent to the following:
\begin{align*}
    0&=  \int_\Omega \frac{u(x)-u(y)}{\vert x-y\vert ^{n+2s}} dy \\
    &= u(x) \int_\Omega \frac{1}{\vert x-y\vert ^{n+2s}} dy - \int_\Omega \frac{u(y)}{\vert x-y\vert ^{n+2s}} dy \\
    &= u(x) \rho(x) - \int_\Omega \frac{u(y)}{\vert x-y\vert ^{n+2s}} dy,\qquad\qquad x\in\Omb.
\end{align*}
This yields the claim and the proof is finished.
\end{proof}

Let \(u \in L^2(\Omega)\). 
We denote by \(u_N\) the extension of \(u\) as follows:

\begin{align}\label{ext-N}
    u_N(x):= \begin{cases} u(x) &\text{ if } x \in \Omega, \\
   \displaystyle \frac{1}{\rho(x)} \int_\Omega \frac{u(y)}{\vert x-y\vert ^{n+2s}}\;dy &\text{ if } x \in \Omb.
    \end{cases}
\end{align}
Since $\partial\Omega$ is a null set, we have that \(u_N\) is well defined for every \(u \in L^2(\Om)\).

It follows from Lemma \ref{Chara_Neumann_boundary} that if \(u_N \in W^{s,2}_\Omega\), then
$\mathcal{N}^s u_N  =0 \text{ a.e. in }  \Omb$.

Next, we give some properties of the  extension $u_N$ that will be used later in the paper.

\begin{lem}\label{cont-ext-N}
Let \(u \in L^2(\Omega)\) and $u_N$ be given in \eqref{ext-N}. 
Then the following assertions hold:
\begin{enumerate}
\item If $u\ge 0$ a.e. in $\Omega$, then $u_N\ge 0$ a.e. in $\RR^n$.

\item If $u=1$ a.e. in $\Omega$, then $u_N=1$ a.e. in $\RR^n$.

\item \(|u_N| \leq \vert u \vert_N\) a.e. in $\RR^n$.

\item If $\Omega$ is of class $C^1$ and $u\in C(\bOm)$, then $u_N\in C(\RR^n)$.
\end{enumerate}
\end{lem}

\begin{proof}
Parts (a) and (b) follow directly from  \eqref{ext-N}.
Part (c) follows directly from the triangle inequality. For Part (d), let $u\in C(\bOm)$. Since $\mathcal N^su_N=0$ a.e. in $\Omb$, it follows from \cite[Proposition 5.2]{SDipierro_XRosOton_EValdinoci_2017a} that $u_N\in C(\RR^n)$. The proof is finished.
\end{proof}

Now we are ready to give a  characterization of the nonlocal Neumann exterior condition.

\begin{thm}\label{thm-38}
Let 
\begin{align}\label{dom-f-aN}
    D(\gota_N):=\Big\{u \in L^2(\Omega):\; u_N\in W^{s,2}_\Omega\Big\}
\end{align}
and \(\gota_N: D(\gota_N) \times D(\gota_N) \rightarrow \mathbb{R}\) be defined by
\begin{align}\label{def-forma1}
    \gota_N(u,v):=\mathcal{E}(u_N,v_N).
\end{align}
 Then \(\gota_N\) is a closed, symmetric and densely defined bilinear form on \(L^2(\Omega) \). The selfadjoint operator $A_N$ on $L^2(\Om)$ associated  with \(\gota_N\) is given by 
\begin{equation}\label{op-AN}
\begin{cases}
\displaystyle D(A_N):=\Big\{u \in L^2(\Omega):\;u_N\in W^{s,2}_\Omega\;\exists\;f\in L^2(\Om) \mbox{ such that }  u_N \mbox{ is a weak solution of } \eqref{NePr} \\
\hfill \mbox{with right hand side } f\Big\},\\
 A_Nu:=f.
 \end{cases}
\end{equation}
\end{thm}

The proof of the theorem uses the following result.

\begin{lem}\label{thm-41} 
Let  \(D(\gota_D)\) and \(D(\gota_N)\) be the spaces defined in \eqref{dom-DC} and  \eqref{dom-f-aN}, respectively.
Then \(D(\gota_D)\subset D(\gota_N)\).
\end{lem}

\begin{proof}
Let \(u \in D(\gota_D)\) and $u_D$ be given by \eqref{uD}.
By definition  \(u_D \in W^{s,2}_\Omega\).  We have to show that \(u_N\in W^{s,2}_\Omega\). That is, we have to prove that
\begin{align*}
    \mathcal{E}(u_N,u_N)
    =\frac{C_{n,s}}{2}\int_\Omega \int_\Omega \frac{( u(x)-u(y))^2}{\vert x-y\vert^{n+2s} } dy \;dx   +C_{n,s}\int_\Omega \int_{\Omc} \frac{(u(x)-u_N(y))^2}{\vert x-y\vert^{n+2s} } dy \;dx 
    < \infty.
\end{align*}
Since \(u \in D(\gota_D)\), we have that
\begin{align*}
    \mathcal{E}(u_D,u_D)  =\frac{C_{n,s}}{2}\int_\Omega \int_\Omega \frac{ (u(x)-u(y))^2}{\vert x-y\vert^{n+2s} } dy \;dx   +C_{n,s}\int_\Omega \int_{\Omc} \frac{u^2(x)}{\vert x-y\vert^{n+2s} } dy \;dx <\infty.
\end{align*}
Obviously \(\mathcal{E}(u_N,u_N)\geq 0\). From this we can deduce that \(\mathcal{E}(u_N,u_N)-\mathcal{E}(u_D,u_D)> - \infty\). On the other hand we have that
\begin{align}\label{M2}
    \mathcal{E}(u_N,u_N)-\mathcal{E}(u_D,u_D) =&C_{n,s} \int_\Omega \int_{\Omc} \frac{-2u(x) u_N(y)+u_N^2(y)}{\vert x-y\vert^{n+2s}}\; dy \;dx \notag\\
    =& C_{n,s}\int_{\Omc}\Big( -2 u_N(y) \int_\Omega \frac{u(x)}{\vert x-y\vert^{n+2s} }\;dx\Big)\;dy\notag \\
    & +C_{n,s} \int_{\Omc}u_N^2(y) \int_\Omega \frac{1}{\vert x-y\vert^{n+2s} }dx\; dy \notag \\
    = &C_{n,s}\int_{\Omc}\Big( -2  \rho(y) u_N^2(y) +u_N^2(y) \rho(y) \Big)\;dy \notag \\
    = &-C_{n,s} \int_{\Omc} u_N^2(y)\rho(y)\; dy
    \leq 0
\end{align}
where we have used that \(\rho\) is a non-negative function. The estimate \eqref{M2} implies that
\begin{align*}
    0\leq \gota_N(u,u)=\mathcal{E}(u_N,u_N)\leq \mathcal{E}(u_D,u_D)=\gota_D(u,u) < \infty.
\end{align*}
The proof is finished.
\end{proof}

\begin{proof}[\bf Proof of Theorem \ref{thm-38}]
Firstly, since the extension operator  \( D(\gota_N) \rightarrow W^{s,2}_\Omega\),  $u\mapsto u_N$  is linear,  and \(\mathcal{E}\) is bilinear and symmetric, we can deduce that \(\gota_N\) is also bilinear and symmetric. 

Secondly, to show that \(\gota_N\) is closed we need to prove that \(D(\gota_N)\) endowed with the norm
\begin{align}\label{norm}
\|u\|_{D(\gota_N)}^2=\gota_N(u,u)+\|u\|_{L^2(\Omega)}^2=\mathcal E(u_N,u_N)+\|u\|_{L^2(\Omega)}^2
\end{align}
is a Hilbert space. Indeed, let \((u_k)_{k \in \mathbb{N}}\subset D(\gota_N)\) be such that

\begin{align*}
    \|u_k-u_m\|_{L^2(\Omega)}^2 + \gota_N(u_k-u_m,u_k-u_m) \rightarrow 0\;\mbox{ as }\; k,m\rightarrow \infty.
\end{align*}
This is the same as
\begin{align*}
    \|(u_k)_N-(u_m)_N\|_{L^2(\Omega)}^2 + \mathcal{E}((u_k)_N-(u_m)_N,(u_k)_N-(u_m)_N) \rightarrow 0 \;\mbox{ as }\; k,m\rightarrow \infty.
\end{align*}
Recall that \(W^{s,2}_\Omega\) endowed with the norm given in \eqref{norm-RV} is a Hilbert space (see Lemma \ref{lem22}). Thus, there is a function \(v \in W^{s,2}_\Omega\) such that \((u_k)_N \rightarrow v\) in $W_\Omega^{s,2}$ as $k\to\infty$. Using Lemma \ref{lem:Nmap}, we can deduce that (passing to a subsequence if necessary)
   $ \mathcal{N}^s v= \lim_{k\rightarrow \infty} \mathcal{N}^s (u_k)_N =0 \text{ a.e. in } \Omb$.
Let us define \(u: = v\vert_\Omega\). Then \cref{Chara_Neumann_boundary} implies that
 \(u_N=v\). Thus,
\begin{align*}
    &\lim_{k\rightarrow \infty}\|u-u_k\|_{L^2(\Omega)}^2+\gota_N(u-u_k,u-u_k) \\
    &= \lim_{k\rightarrow \infty}
    \|(u_k)_N-u_N\|_{L^2(\Omega)}^2 + \mathcal{E}((u_k)_N-u_N,(u_k)_N-u_N) =0.
\end{align*}
Hence, \(D(\gota_N)\) is complete and we have shown that the form \(\gota_N\) is closed.

By Lemma \ref{thm-41} $D(\gota_D)\subset D(\gota_N)$, and since $D(\gota_D)$ is dense in $L^2(\Om)$ (by Theorem \ref{thm-31}), we have that $D(\gota_N)$ is dense in $L^2(\Om)$. We have shown that $\gota_N$ is densely defined.

Thirdly, let \(B\) be the selfadjoint operator on $L^2(\Om)$ associated with \(\gota_N\) in the sense of \eqref{op-A}. 
We show that $B=A_N$.
Let $u\in D(A_N)$ and set $f:=A_Nu$.  Then by definition $u_N\in W_\Omega^{s,2}$. Thus, $u\in D(\gota_N)$.
Since $A_Nu=f$ in $\Omega$ and $u_N$ is a weak solution of \eqref{NePr} with right hand side \(f\) (by the  definition of $A_N$), we have that  
\begin{align*}
\int_{\Omega}wf\;dx=\mathcal E(u_N,w)
\end{align*}
for every $w\in W_\Omega^{s,2}$. In particular, we have that
\begin{align*}
\int_{\Omega}vf\;dx=\mathcal E(u_N,v_N)=\gota_N(u,v)
\end{align*}
for every $v\in D(\gota_N)$. Thus, $u\in D(B)$ and $A_Nu=Bu$. We have shown \(A_N \subseteq B\).  It has been shown in \cite[Theorem 3.11]{SDipierro_XRosOton_EValdinoci_2017a} that the operator \(A_N\) is selfadjoint (more precisely, it is closed, has a real spectrum and its eigenfunctions form an orthogonal system in \(L^2(\Om)\)). Since \(B\) is by definition a selfadjoint operator and \(A_N \subseteq B\), we can deduce that \(A_N=B\) (since selfadjoint operators cannot be subsets of each other).
We have shown that \(A_N=B\) and the proof is finished.
\end{proof}

\begin{remark}
{\em The operator $A_N$ is the realization in $L^2(\Omega)$ of $(-\Delta)^s$ with  the nonlocal Neumann exterior condition.}
\end{remark}

It is worth mentioning the following characterization of $D(\gota_N)$.
\begin{lem}\label{lem-inf}
Let $D(\gota_N)$ be the space defined in \eqref{dom-f-aN}. Then
\begin{align}\label{def-formaN}
D(\gota_N)=\Big\{u|_{\Omega}:\; u\in W_\Omega^{s,2}\Big\}.
\end{align}
\end{lem}

\begin{proof}
Denote by $D$ the right hand side of \eqref{def-formaN}. It is clear that $D(\gota_N)\subseteq D$. Now let $v\in D$. Then, $v=u|_{\Omega}$ for some $u\in W_\Omega^{s,2}$. We have to show that $\mathcal E(v_N,v_N)<\infty$. Calculating we get that
\begin{align*}
\mathcal E(v_N,v_N)-\mathcal E(u,u)=&C_{n,s}\int_{\Om}\int_{\Omc}\frac{(v(x)-v_N(y))^2-(v(x)-u(y))^2}{|x-y|^{n+2s}}\;dy\;dx\\
=&C_{n,s}\int_{\Omc}\Big(-\rho(y)v_N^2(y)+2\rho(y) v_N(y)u(y)-\rho(y) u^2(y)\Big)\;dy\\
=&-C_{n,s}\int_{\Omc}\rho(y)\left(v_N(y)-u(y)\right)^2\;dy\le 0.
\end{align*}
We have shown that $\mathcal E(v_N,v_N)\le \mathcal E(u,u)<\infty$ and the proof is finished.
\end{proof}

We have the following result as a direct consequence of the proof of Lemma \ref{lem-inf}.

\begin{proposition}\label{prop-Neumann}
Let $(\gota_N,D(\gota_N))$ be the form defined in \eqref{dom-f-aN}-\eqref{def-forma1}. Then
\begin{align}\label{inf-N}
\gota_N(u,u):=\mathcal E(u_N,u_N)=\inf\Big\{\mathcal E(v,v):\; v\in W_\Omega^{s,2}\mbox{ and } v|_{\Omega}=u\Big\}.
\end{align}
In other words, for $u\in L^2(\Omega)$, we have that $u_N$ is the smallest extension in $W_\Om^{s,2}$ with respect to the $W_\Om^{s,2}$-norm, or equivalently, the infimum in the right hand side of \eqref{inf-N} is attained at $u_N$.
\end{proposition}

Denote by $T_N=(e^{-tA_N})_{t\ge 0}$ the semigroup on $L^2(\Omega)$ generated by $-A_N$. 
 
\begin{thm} \label{Neumann_positive}
The semigroup $T_N$ is positivity-preserving.
\end{thm}

\begin{proof}
Let \(u \in D(\gota_N)\). 
We want to show that \(\gota_N(\vert u \vert,|u|)\le\gota_N(u,u)\). Notice that for  $v,w\in W_\Om^{s,2}$, $\mathcal E(v,w)$ is a sum of integrals over \(\Omega \times \Omega\) and over  \(\Omega \times (\Omc)\) (see \eqref{form-EN}). Firstly, let us inspect the \(\Omega \times \Omega\) part. We define
\begin{align}\label{F-Omega}
    \mathcal{E}_\Omega(v,w):=\frac{C_{n,s}}{2} \int_\Omega \int_\Omega \frac{(v(x)-v(y))(w(x)-w(y))}{\vert x-y\vert^{n+2s}} dx\; dy.
\end{align}
The reverse triangle inequality yields 
\begin{align}\label{N1}
    \mathcal{E}_\Omega (\vert u\vert,|u| ) &=\frac{C_{n,s}}{2}\int_\Omega \int_\Omega \frac{\big\vert \vert u(x)\vert -\vert u(y)\vert \big \vert^2}{\vert x-y\vert^{n+2s}} dx\; dy \notag\\
    &\leq \frac{C_{n,s}}{2}\int_\Omega \int_\Omega \frac{\big\vert u(x) - u(y) \big \vert^2}{\vert x-y\vert^{n+2s}} dx\; dy
 =\mathcal{E}_\Omega(u,u).
\end{align}

Secondly, for the \(\Omega \times (\Omc)\) part we have that
\begin{align}\label{N2}
    &C_{n,s}\int_\Omega \int_{\Omc} \frac{(\vert u \vert (x)-\vert u \vert _N(y))^2-(u(x)-u_N(y))^2}{\vert x-y\vert ^{n+2s}} dy \; dx\notag\\
    &=C_{n,s}\int_{\Omc}\Big(-2 \rho(y)  \vert u \vert _N^2(y) + \rho(y) \vert u \vert_N^2(y) + 2  \rho(y) u_N^2(y)-\rho(y) u_N^2(y) \Big)\;dy\notag\\
    &=C_{n,s} \int_{\Omc} \Big(- \rho(y) \vert u \vert_N^2 (y) +  \rho(y) u_N^2(y)\Big)\;dy\notag\\
    &=C_{n,s} \int_{\Omc}\rho(y)\Big (u_N^2(y)-\vert u \vert_N^2(y)\Big)\;dy.
\end{align}
Using the assertion \((c)\) in Lemma \ref{cont-ext-N} we can deduce from \eqref{N2} that
\begin{align}\label{N3}
     &C_{n,s}\int_\Omega \int_{\Omc} \frac{(\vert u \vert (x)-\vert u \vert _N(y))^2-(u(x)-u_N(y))^2}{\vert x-y\vert ^{n+2s}} dy \; dx\notag\\
     &=C_{n,s} \int_{\Omc}\rho(y)\Big (u_N^2(y)-\vert u \vert_N^2(y)\Big)\;dy 
    \leq 0.
\end{align}
Combining \eqref{N1}-\eqref{N3} we get that \(\gota_N(\vert u \vert,|u|) \leq \gota_N(u,u)\).
The proof is finished.
\end{proof}

\subsection{The nonlocal Robin exterior condition}
We recall that both Dirichlet and the nonlocal Neumann exterior conditions where realized by some kind of extension to $\RR^n$ of functions defined in  \(\Omega\). Here we also need to find an appropriate extension.  Firstly, let $f\in L^2(\Omega)$, $\beta\in L^1(\Omc)$ and consider the following Robin problem:
\begin{equation}\label{RoPr}
(-\Delta)^su=f\;\;\mbox{ in }\;\Omega,\;\;\; \mathcal N^su+\beta u=0\;\mbox{ in }\;\Omb.
\end{equation}
By a weak solution to \eqref{RoPr}, we mean a function $u\in W_{\beta,\Omega}^{s,2}$ such that
\begin{align} \label{RoPrWe}
\mathcal E(u,v)+\int_{\Omc}\beta uv\;dx=\int_{\Omega}fv\;dx
\end{align}
for every $v\in W_{\beta,\Omega}^{s,2}$. Here also the existence and uniqueness of weak solutions is easy to prove.
As in the case of the Neuman problem, to the best of our knowledge, the following is an open problem: Let \(u\) be a weak solution of the  Robin problem \eqref{RoPr}. Is \(u\) a strong solution in the sense that \eqref{RoPr} holds almost everywhere? As in the Neumann case,  in \eqref{RoPr}, it is easy to see that the exterior condition $ \mathcal N^su+\beta u=0$ in $\Omb$ holds almost everywhere for weak solutions.

We start with the following result.

\begin{lem}\label{lem-314}
Let  \(\beta \in L^1(\Omc)\) be a fixed non-negative function and $u\in W_\Omega^{s,2}$. Then
\begin{align}\label{RBC}
\mathcal N^su(x)+\beta(x)u(x)=0,\;\;x\in\Omb,
\end{align}
if and only if
\begin{align}\label{ex-u}
    u(x)=\frac{C_{n,s}\rho(x)}{C_{n,s}\rho(x)+\beta(x)}u_N(x)=\frac{C_{n,s}}{C_{n,s}\rho(x)+\beta(x)}\int_\Omega \frac{u(y)}{\vert x-y\vert^{n+2s}} dy,\;\;\;x\in\Omb,
\end{align}
where we recall that $\rho(x)$ has been defined in \eqref{rho} and $u_N$ is given in \eqref{ext-N}.
\end{lem}

\begin{proof}
Let  \(u \in W^{s,2}_\Omega\). A simple calculation yields that the  condition \eqref{RBC}, that is,
\begin{align*}
    0&=\mathcal{N}^s u(x)+ \beta(x) u(x) \\
    &=C_{n,s}u(x)\rho(x)-C_{n,s}\rho(x) u_N(x) + \beta(x) u(x)\\
     &=\Big(C_{n,s}\rho(x)+\beta(x)\Big)u(x)-C_{n,s}\rho(x) u_N(x) ,\;\;x\in\Omb,
\end{align*}
is equivalent to \eqref{ex-u}.
The proof is finished.
\end{proof}

For a function $u\in L^2(\Om)$ we define its extension \(u_R\) as follows:
\begin{align*}
    u_R(x):=\begin{cases} u(x) &\text{ if } x \in \Omega, \\
 \displaystyle   \frac{C_{n,s}}{C_{n,s}\rho(x)+\beta(x)}\int_\Omega \frac{u(y)}{\vert x-y\vert^{n+2s}} dy &\text{ if } x\in\Omb.
    \end{cases}
\end{align*}
As for the Neumann case, we have that $u_R$ is well defined for every $u\in L^2(\Omega)$.

\begin{remark}\label{rem-uR}
{\em Let $u\in W_\Om^{s,2}$. By Lemma \ref{lem-314} we have that $u_R$ satisfies \eqref{RBC}. The identity \eqref{RBC} is called the nonlocal Robin exterior condition.  If $u\ge 0$ a.e. in $\Om$, then $u_R\ge 0$ a.e. in $\RR^n$. In addition, it follows from \eqref{ex-u} that 
\begin{align}\label{N-to-R}
u_R(x)=\frac{C_{n,s}\rho(x)}{C_{n,s}\rho(x)+\beta(x)}u_N(x)\;\mbox{ for a.e. } \;x\in\Omb.
\end{align}
Note that  $|u_R| \leq |u|_R$ a.e. in $\RR^n$ as in Lemma \ref{cont-ext-N} for the Neumann exterior condition.
}
\end{remark}

Before characterizing our operator we need some preparations. 

\begin{lem}\label{lem-35}
Let $u\in W^{s,2}_\Omega$ and  \((u_k)_{k\in \mathbb{N}}\subset W^{s,2}_\Omega\) be such that \(u_k \rightarrow u\) in $W_\Omega^{s,2}$ as $k\to\infty$. Then there is a subsequence, that we still denote by \((u_k)_{k\in\NN}\), such that  \(u_k \rightarrow u\) pointwise a.e. in $\RR^n$ as $k\to\infty$.
\end{lem}

\begin{proof}
Let \((u_k)_{k\in \mathbb{N}}\) and \(u\) be as in the statement of the lemma. By definition \(u_k\rightarrow u\) in \(L^2(\Omega)\) as $k\to\infty$. Hence,  passing to a subsequence if necessary,  \(u_k\to u\) a.e. in \(\Omega\) as $k\to\infty$. On the other hand we have that
\begin{align*}
    \int_\Omega \int_{\Omc}\frac{\big|(u_k(x)-u(x))-(u_k(y)-u(y))\big|^2}{\vert x-y\vert^{n+2s}} dy \;dx \rightarrow 0\;\mbox{ as }\;k\to\infty.
\end{align*}
Hence, passing to a subsequence if necessary,
    $(u_k(x)-u(x))-(u_k(y)-u(y))\rightarrow 0$ for a.e.  \(x \in \Omega\) and a.e.  \(y \in \Omc\), as $k\to\infty$.
Therefore,
  $  0=\lim_{k\rightarrow \infty}\left(u_k(x)-u(x)\right)=\lim_{k\rightarrow \infty}\left(u_k(y)-u(y)\right)$
for a.e. \(x \in \Omega\) and a.e. \(y \in \Omc\). The proof is finished.
\end{proof}

Now we introduce the realization in $L^2(\Om)$ of $(-\Delta)^s$ with the nonlocal Robin exterior condition.

\begin{thm}\label{thm-311}
Let $\beta\in L^1(\Omc)$ be a non-negative function,
\begin{align}\label{form-Rob1}
    D(\gota_R):=\Big\{ u \in L^2(\Om):\;u_R \in W^{s,2}_{\beta,\Omega}\Big \},
\end{align}
and let  \(\gota_R :D(\gota_R)\times D(\gota_R) \rightarrow \mathbb{R}\) be given by
\begin{align}\label{form-Rob2}
    \gota_R(u,v):=\mathcal{E}(u_R,v_R)+\int_{\Omc} \beta u_R v_R\;dx.
\end{align}
Then  \(\gota_R\) is a closed, symmetric and densely defined bilinear form on $L^2(\Om)$.  The  selfadjoint operator $A_R$ associated with \(\gota_R\) is given by 

\begin{equation*}
\begin{cases}
\displaystyle D(A_R):=
   \Big\{ u \in L^2(\Omega):
    u_R \in W^{s,2}_{\beta,\Omega}\;\;\exists\; f\in L^2(\Omega)\mbox{ such that }\; u_R \mbox{ is a weak solution of } \eqref{RoPr} \\
    \hfill\text{ with right hand side } f\Big\},\\
    A_R u:= f.
    \end{cases}
\end{equation*}
\end{thm}

\begin{proof}
Let \(u\in D(\gota_D)\). It follows from Lemma \ref{thm-41} that \(u \in D(\gota_N)\). This shows that the extensions $u_N$ and $u_R$ are well defined. Obviously, \(D(\gota_R)=\{u \in L^2(\Omega): \gota_R(u,u)< \infty\}\). Calculating we get that
\begin{align*}
    \gota_R(u,u)-\gota_D(u,u) &= C_{n,s} \int_\Omega \int_{\Omc}\frac{-2 u(x) u_R(y)+u_R^2(y)}{\vert x-y\vert^{n+2s}} \;dy \;dx + \int_{\Omc}\beta(y) u_R^2(y)\; dy \\
    &=  C_{n,s} \int_{\Omc}\Big( -2 \rho(y) u_N(y) u_R(y) + \rho(y) u_R^2(y) + \frac{\beta(y)}{C_{n,s}} u_R^2(y)\Big) dy\\
    &= C_{n,s} \int_{\Omc}\Big( -2 \rho(y) \frac{C_{n,s}\rho(y)+\beta(y)}{C_{n,s}\rho(y)}+ \rho(y)  + \frac{\beta(y)}{C_{n,s}} \Big)u_R^2(y)\; dy\\
    &= -C_{n,s} \int_{\Omc}  \Big(\rho(y) +\frac{ \beta(y) }{C_{n,s}} \Big)u_R^2(y)\;dy
     \leq 0,
\end{align*}
where in the third equality we have used \eqref{N-to-R}.
This implies that  $0\leq \gota_R(u,u) \leq \gota_D(u,u) < \infty$.
Therefore, $D(\gota_D)\subset D(\gota_R)$. Thus, $D(\gota_R)$ is dense in $L^2(\Omega)$ (since $D(\gota_D)$ is dense in $L^2(\Omega)$ by Theorem \ref{thm-31}).

Next, let \((u_k)_{k\in \mathbb{N}}\subset D(\gota_R)\) be such that
\begin{align*}
    \gota_R(u_k-u_m,u_k-u_m) +\|u_k-u_m\|_{L^2(\Om)}^2\rightarrow 0\;\mbox{ as }\; k,m\to\infty.
\end{align*}
This is the same as 
\begin{align*}
    \mathcal{E}((u_k)_R-(u_m)_R,(u_k)_R-(u_m)_R) + \int_{\Omc} \beta \big((u_k)_R-(u_m)_R \big)^2\;dx +\|u_k-u_m\|_{L^2(\Om)}^2\rightarrow 0
\end{align*}
 as $k,m\to\infty$. Since \(W^{s,2}_{\beta,\Omega}\) is a Hilbert space, it follows that  there exists a function \(v\in W^{s,2}_{\beta,\Omega}\) such that \((u_k)_R \rightarrow v\) in \(W^{s,2}_{\beta,\Omega}\) as $k\to\infty$.  Using Lemma \ref{lem:Nmap} we get that
\begin{align*}
    \mathcal{N}^s (u_k)_R \rightarrow \mathcal{N}^s v = f\; \mbox{ in } \;L_{\rm loc}^2(\Omc) \;\mbox{ as }\; k\to\infty.
\end{align*}
This implies that (passing to a subsequence if necessary) $\beta (u_k)_R \rightarrow -f$ a.e. in \(\Omc\) as $k\to\infty$. Since  $(u_k)_R$ converges to $v$ in \(W^{s,2}_{\beta,\Omega}\) as $k\to\infty$, we can deduce that (passing to a subsequence if necessary)  $(u_k)_R$ converges a.e.  to $v$ in $\RR^n$ as $k\to\infty$ (by Lemma \ref{lem-35}). This implies that
   $ \beta (u_k)_R \rightarrow \beta v\; \mbox{ a.e. in }\;
    \Omc \;\mbox{ as }\; k\to\infty$.
Furthermore, since
     $\beta \big((u_k)_R-u_R \big)^2 \rightarrow 0$ in $L^1(\Omc)$ as $k\to\infty$,
it follows that \((u_k)_R \rightarrow g \) in \(L^2(\Omc,\beta dx)\) as $k\to\infty$. But again the pointwise convergence (passing to a subsequence if necessary) shows that \(g=v\) a.e. in $\Omc$.
Hence, \(v\) fulfills the nonlocal Robin exterior condition \eqref{RBC} and  
\begin{align*}
    \mathcal{E}((u_k)_R-v,(u_k)_R-v) +   \big\|\sqrt{\beta}\left((u_k)_R-v \right)\big\|_{L^2(\Omc)}^2 \rightarrow 0\;\mbox{ as }\;k\to\infty.
\end{align*}
Let \(u:= v \vert_\Omega\). Then, \(u_R=v\) and
\begin{align*}
 \gota_R(u_k-u,u_k-u) +\|u_k-u\|_{L^2(\Om)}^2 \rightarrow 0\;\mbox{ as }\;k\to\infty.
\end{align*}
We have shown that \(\gota_R\) is closed. 

Proceeding exactly as in \cite[Theorem 3.11]{SDipierro_XRosOton_EValdinoci_2017a}  for the Neumann case, we can deduce that the operator $A_R$ is selfadjoint.
Now, let \(B\) be the operator associated with \(\gota_R\) in the sense of \eqref{op-A}. The claim that \(A_R=B\) follows similarly as in the case of the Neumann exterior condition. The proof is finished.
\end{proof}

The following is the variant of Lemma \ref{lem-inf} for the form $\gota_R$.

\begin{lem}\label{lem-inf-R}
Let $D(\gota_R)$ be the space defined in \eqref{form-Rob1}. Then
\begin{align}\label{def-formaR}
D(\gota_R)=\Big\{u|_{\Omega}:\; u\in W_{\beta,\Omega}^{s,2} \Big\}.
\end{align}
\end{lem}

\begin{proof}
Let $D$ denote the right hand side of \eqref{def-formaR}. It is clear that $D(\gota_R)\subseteq D$. 

Conversely, let $v\in D$. Then, $v=u|_{\Omega}$ for some $u\in W_{\beta,\Omega}^{s,2}$. We have to show that 
\begin{align*}
\mathcal E(v_R,v_R)+\int_{\Omc}\beta|v_R|^2\;dx<\infty.
\end{align*}
Calculating we get that
\begin{align}\label{M1}
\mathcal E(v_R,v_R)&+\int_{\Omc}\beta(y)|v_R(y)|^2\;dy-\mathcal E(u,u)-\int_{\Omc}\beta(y)|u(y)|^2\;dy\notag\\
 =&C_{n,s}\int_{\Om}\int_{\Omc}\frac{(v(x)-v_R(y))^2-(v(x)-u(y))^2}{|x-y|^{n+2s}}\;dy\;dx\notag\\
&+\int_{\Omc}\beta(y)\left(|u_R(y)|^2-|u(y)|^2\right)\;dy\notag\\
=&C_{n,s}\int_{\Omc}\left(\rho(y)v_R^2(y)-2\rho(y)v_N(y)v_R(y)+2\rho(y) v_N(y)u(y)-\rho(y) u^2(y)\right)\;dy\notag\\
&+\int_{\Omc}\beta(y)\left(|v_R(y)|^2-|u(y)|^2\right)\;dy.
\end{align}
Using the fact that (by \eqref{N-to-R})
\begin{align*}
v_N(y)=\left(1+\frac{\beta(y)}{C_{n,s}\rho(y)}\right)v_R(y)\;\mbox{ for a.e. }\; y\in\Omb,
\end{align*}
we get from \eqref{M1} that
\begin{align*}
\mathcal E(v_R,v_R)&+\int_{\Omc}\beta(y)|v_R(y)|^2\;dy-\mathcal E(u,u)-\int_{\Omc}\beta(y)|u(y)|^2\;dy\\
=&-C_{n,s}\int_{\Omc}\rho(y)\left(1+\frac{\beta(y)}{C_{n,s}\rho(y)}\right)\Big(v_R(y)-u(y)\Big)^2\;dy
\le 0.
\end{align*}
We have shown that 
\begin{align*}
\mathcal E(v_R,v_R)+\int_{\Omc}\beta(x)|v_R(x)|^2\;dx\le \mathcal E(u,u)+\int_{\Omc}\beta(x)|u(x)|^2\;dx<\infty.
\end{align*}
 The proof is finished
\end{proof}

Here also we have the following result as a direct consequence of the proof of Lemma \ref{lem-inf-R}.

\begin{proposition}\label{prop-Robin}
Let $(\gota_R,D(\gota_R))$ be the form defined in \eqref{form-Rob1}-\eqref{form-Rob2}. Then, for $u\in D(\gota_R)$ we have 
\begin{align}\label{W1}
\gota_R(u,u):=\mathcal E(u_R,u_R)=\inf\Big\{\mathcal E(v,v)+\int_{\Omc}\beta(x)|v(x)|^2\;dx:\; v\in W_{\beta,\Omega}^{s,2}\mbox{ and } v|_{\Omega}=u\Big\}.
\end{align}
In other words, if $u\in L^2(\Omega)$, then $u_R$ is the smallest extension in $W_{\beta,\Omega}^{s,2}$ with respect to the $W_{\beta,\Omega}^{s,2}$-norm, or equivalently, the infimum in \eqref{W1} is attained at $u_R$.
\end{proposition}

Next, denote by $T_R=(e^{-tA_R})_{t\ge 0}$ the semigroup on $L^2(\Omega)$ generated by $-A_R$. 

\begin{thm}\label{thm-320}
The semigroup $T_R$ is positivity-preserving.
\end{thm}

\begin{proof}
Let \(u \in D(\gota_R)\). 
We want to show that \(\gota_R(\vert u \vert,|u|)\le \gota_R(u,u)\). As in the proof of \cref{Neumann_positive}, we have a sum of integrals over \(\Omega \times \Omega\) and over \(\Omega \times (\Omc)\). Firstly, let us inspect the \(\Omega \times \Omega\) part. Let $\mathcal E_\Omega$ be as in \eqref{F-Omega}. Then proceeding as in \eqref{N1} we get that
\begin{align}\label{R1}
    \mathcal{E}_\Omega (\vert u\vert,|u| ) 
    \le \mathcal{E}_\Omega(u,u).
\end{align}

Secondly, for the \(\Omega \times (\Omc)\) part we have 
\begin{align}\label{R2}
    &C_{n,s}\int_\Omega \int_{\Omc} \frac{(\vert u \vert (x)-\vert u \vert _R(y))^2-(u(x)-u_R(y))^2}{\vert x-y\vert ^{n+2s}} dy \; dx\notag\\
    &  \quad + \int_{\Omc} \Big(\beta(y) \vert u \vert_R^2-\beta(y) u_R(y)^2 \Big)dy\notag\\
    &=C_{n,s}\int_\Omega \int_{\Omc}\frac{\vert u(x) \vert ^2-2 \vert u \vert (x) \vert u \vert _R(y) + \vert u \vert_R^2(y) -u^2(x) + 2 u(x)u_R(y)-u_R^2(y)}{\vert x-y\vert ^{n+2s}}dy \; dx\notag\\
    &  \quad + \int_{\Omc}\Big( \beta(y) \vert u \vert_R^2(y)-\beta(y) u_R^2(y)\Big)\; dy\notag\\
    &=C_{n,s}\int_{\Omc}\Big(-2 \rho(y) \vert u \vert_N (y) \vert u \vert _R(y) + \rho(y) \vert u \vert_R^2(y) + 2  \rho(y) u_N(y) u_R(y)-\rho(y) u_R^2(y)\Big)\; dy\notag\\
    &  \quad + \int_{\Omc}\Big( \beta(y) \vert u \vert_R^2(y)-\beta(y) u_R^2(y)\Big)\; dy\notag\\
    &=C_{n,s}\int_{\Omc}\Big(-2 \rho(y) \Big(1+\frac{\beta(y)}{C_{n,s}\rho(y)}\Big)\vert u \vert_R^2 (y)  + \rho(y) \vert u \vert_R^2(y)\Big)\;dy \notag\\
    & \quad+\int_{\Omc}\Big( 2  \rho(y)\Big(1+\frac{\beta(y)}{C_{n,s}\rho(y)}\Big) u_R^2(y)-\rho(y) u_R^2(y) \Big)\;dy\notag\\
    &  \quad + C_{n,s}\int_{\Omc}\Big( \frac{\beta(y)}{C_{n,s}} \vert u \vert_R^2(y)-\frac{\beta(y)}{C_{n,s}} u_R^2(y)\Big)\; dy\notag\\
    &=C_{n,s} \int_{\Omc} \Big(- \rho(y) \vert u \vert_R^2(y) -\frac{ \beta(y)}{C_{n,s}} \vert u \vert_R^2(y) + \rho(y) u_R^2(y) + \frac{\beta(y)}{C_{n,s}} u_R^2(y) \Big)\; dy\notag\\
    &=C_{n,s} \int_{\Omc} \Big(\rho(y)+\frac{\beta(y)}{C_{n,s}}\Big)\Big(u_R^2(y)-\vert u \vert_R^2(y)\Big)\; dy
     \leq 0,
\end{align}
where we have used \eqref{N-to-R} and the last inequality follows from Remark \ref{rem-uR}. Combining  \eqref{R1}-\eqref{R2} we get that \(\gota_R(\vert u \vert,|u|) \leq \gota_R(u,u)\). The proof is finished.
\end{proof}

\section{Some domination results}\label{sec-dom}

In this section we give some results on domination of the semigroups constructed in Section \ref{sec-3}. 

First, we consider the semigroups $T_D$ and  $T_N$.

\begin{thm}\label{dom_Dirichlet_Neumann}
Let $T_D$ and $T_N$ be the semigroups given in Theorems \ref{thm-33} and  \ref{Neumann_positive}, respectively.
Then,
\begin{align}\label{dom1}
    0\le T_D\leq T_N
\end{align}
in the sense of \eqref{dom}.
\end{thm}

\begin{proof}
We have already shown in Theorems \ref{thm-33} and \ref{Neumann_positive} that $T_D$ and $T_N$ are positivity-preserving. It remains to verify that the conditions in Theorem \ref{thm-29}(ii) are satisfied. \\
    
{\bf Step 1}:  Recall that  \(D(\gota_D) \subset D(\gota_N)\) (by Lemma \ref{thm-41}). 
Next, let  \(u \in D(\gota_D)\) and \(v \in D(\gota_N)\) be such that $0\leq v \leq u$. 
   Then,
    \begin{align*}
        \gota_D(v,v)&= \frac{C_{n,s}}{2}\int_\Omega \int_\Omega \frac{ (v(x)-v(y))^2}{\vert x-y\vert^{n+2s} } dy \;dx   +C_{n,s}\int_\Omega \int_{\Omc} \frac{v^2(x)}{\vert x-y\vert^{n+2s} } dy \;dx \\
        & \leq \gota_N(v,v) +C_{n,s}\int_\Omega \int_{\Omc} \frac{v^2(x)}{\vert x-y\vert^{n+2s} } dy \;dx \\
        &\leq \gota_N(v,v) +C_{n,s}\int_\Omega \int_{\Omc} \frac{u^2(x)}{\vert x-y\vert^{n+2s} } dy \;dx \\
        & \leq \gota_N(v,v)+\gota_D(u,u) 
        < \infty.
    \end{align*}
Hence, \(v \in D(\gota_D)\) and we have shown that $D(\gota_D)$ is an ideal in $D(\gota_N)$.\\
 
{\bf Step 2}:  Let \(0\le u,v \in D(\gota_D)\). Then, $0\le u_N,v_N$ by Lemma \ref{cont-ext-N}(a).
Calculating we get that
    \begin{align*}
        \gota_N(u,v)-\gota_D(u,v)&= C_{n,s}\int_\Omega \int_{\Omc} \frac{-u(x) v_N(y)- u_N(y) v(x)+u_N(y) v_N(y)}{\vert x-y\vert^{n+2s} } dy \;dx \\
        &= -C_{n,s}\int_{\Omc} \rho(y) u_N(y) v_N(y) dy \leq 0.
    \end{align*}
 Thus, $\gota_N(u,v)\le \gota_D(u,v)$. \\

{\bf Step 3}: Finally, it follows from Theorem \ref{thm-29} that \eqref{dom1} holds, The proof is finished.
\end{proof}

Now we consider the semigroups  $T_D$, $T_R$ and $T_N$.

\begin{thm}\label{thm-43}
Let $T_D$, $T_N$ and $T_R$ be the semigroups given in Theorems \ref{thm-33},  \ref{Neumann_positive} and \ref{thm-320}, respectively.
Then,
\begin{align}\label{dom2}    
0\le T_D \leq T_R \leq T_N
 \end{align}
 in the sense of \eqref{dom}.
\end{thm}

\begin{proof} 
We prove the result in three steps.\\

{\bf Step 1}: We show that 

\begin{align}\label{dom2-1}
0\le T_D\leq T_R. 
\end{align}
We have already shown in Theorems \ref{thm-33} and \ref{thm-320} that $T_D$ and $T_R$ are positivity-preserving.
 We claim that  \(D(\gota_D)\) is an ideal in  \(D(\gota_R)\).  Indeed, \(D(\gota_D) \subset D(\gota_R)\) by Theorem \ref{thm-311}. Proceeding exactly as in the proof of \cref{dom_Dirichlet_Neumann} we can easily deduce that if  \(u \in D(\gota_D)\) and \(v \in D(\gota_R)\) are such that \(0\leq v \leq u\), then  \(v \in D(\gota_D)\). We have shown the claim.
 
  Next, let \(0\le u,v \in D(\gota_D)\). Then, $0\le u_R,v_R$ by Remark \ref{rem-uR}. Calculating and using \eqref{N-to-R} we get that
    \begin{align*}
        &\gota_R(u,v)-\gota_D(u,v)\\
        = &C_{n,s} \int_\Omega \int_{\Omc}\frac{- u(x) v_R(y) -u_R(y) v(x)+u_R(y) v_R(y)}{\vert x-y\vert^{n+2s}} dy \;dx
       + \int_{\Omc}\beta(y) u_R(y) v_R(y) dy \\
    =& C_{n,s} \int_{\Omc} \Big(- \rho(y) u_N(y) v_R(y) - \rho(y) v_N(y) u_R(y) + \rho(y) u_R(y) v_R(y)
   + \frac{\beta(y)}{C_{n,s}} u_R(y) v_R(y)\Big)\; dy\\
    =&C_{n,s} \int_{\Omc} \Big(-2 \rho(y) \frac{C_{n,s}\rho(y)+\beta(y)}{C_{n,s}\rho(y)}u_R(y) v_R(y) + \rho(y) u_R(y) v_R(y) + \frac{\beta(y)}{C_{n,s}} u_R(y) v_R(y)\Big)\; dy\\
    =&-C_{n,s} \int_{\Omc} \Big( \rho(y)u_R(y) v_R(y) + \frac{ \beta(y)}{C_{n,s}} u_R(y)v_R(y) \Big)dy 
    \leq 0.
    \end{align*}
 Hence, $\gota_R(u,v)\le \gota_D(u,v)$. It follows from Theorem \ref{thm-29} that \eqref{dom2-1} holds.\\
 
{\bf Step 2}: Here we show that
\begin{align}\label{dom-2-2}
0\le  T_R \leq T_N. 
 \end{align}
 Firstly, we claim that \(D(\gota_R)\) is an ideal in  \(D(\gota_N)\). Indeed, let \(u \in D(a_R)\).  Calculating and using \eqref{N-to-R} again we get that
    \begin{align}\label{mw1}
        &\gota_N(u,u)-\gota_R(u,u)\notag\\
        &= C_{n,s} \int_\Omega \int_{\Omc} 
        \frac{(u(x)-u_N(y))^2-(u(x)-u_R(y))^2}{\vert x-y\vert ^{n+2s}} dy \; dx  - \int_{\Omc} \beta(y) u_R^2(y) dy\notag \\
        &= C_{n,s}\int_\Omega \int_{\Omc} 
        \frac{\Big(u(x)-u_R(y)-\frac{\beta(y)}{C_{n,s}\rho(y)}u_R(y)\Big)^2-(u(x)-u_R(y))}{\vert x-y\vert ^{n+2s}} dy \; dx  - \int_{\Omc} \beta(y) u_R^2(y)\; dy\notag \\
        &= C_{n,s}\int_\Omega \int_{\Omc} 
        \frac{ -2 \Big(u(x)-u_R(y)\Big) \frac{\beta(y)}{C_{n,s}\rho(y)} u_R(y) + \frac{\beta^2(y)}{C_{n,s}^2\rho^2(y)} u_R^2(y) }{\vert x-y\vert ^{n+2s}} dy \; dx - \int_{\Omc} \beta(y) u_R^2(y) dy\notag \\
        &=C_{n,s} \int_{\Omc} 
         \left(-2 \rho(y) u_N(y) \frac{\beta(y)}{C_{n,s}\rho(y)} u_R(y) + 2 \rho(y) \frac{\beta(y)}{C_{n,s}\rho(y)} u_R^2(y) + \rho(y) \frac{\beta^2(y)}{C_{n,s}^2\rho^2(y)} u_R^2(y)\right)dy\notag\\
         &=C_{n,s} \int_{\Omc} 
        \left( -2 \frac{\beta(y)}{C_{n,s}} \Big(1+\frac{\beta(y)}{C_{n,s}\rho(y)}\Big)  + 2 \frac{\beta (y)}{C_{n,s}}  +  \frac{\beta^2(y)}{C_{n,s}^2\rho(y)} +\frac{\beta(y)}{C_{n,s}} \right)u_R^2(y)\; dy\notag\\
        &=-C_{n,s} \int_{\Omc} 
           \left( \frac{\beta^2(y)}{C_{n,s}^2\rho(y)}   +\frac{ \beta(y)}{C_{n,s}}  \right)u_R^2(y)\;dy 
          \leq 0.
    \end{align}
Therefore, \(0\leq \gota_N(u,u) \leq \gota_R(u,u)\) which implies that \(u \in D(\gota_N)\) . We have shown that $D(\gota_R)\subset D(\gota_N)$.
    Next,  let \(u \in D(\gota_R)\) and \(v \in D(\gota_N)\) be such that \(0\leq v \leq u\). We have to show that $v\in D(\gota_R)$. It follows from \eqref{mw1} that   
    \begin{align}\label{mjw}
       \gota_R(v,v) &= \mathcal{E}(v_N,v_N)+C_{n,s} \int_{\Omc} \left( \frac{\beta^2(y)}{C_{n,s}^2\rho(y)} + \frac{\beta(y)}{C_{n,s}} \right)v_R^2(y)\;dy \\
        & \leq \gota_N(v,v) + \gota_R(u,u) 
        < \infty.\notag
    \end{align}
    Thus, \(v \in D(\gota_R)\) and the proof of the claim is complete.
    
   Secondly, let \(0\le u,v \in D(\gota_R)\). A similar calculation yields    
    \begin{align*}
        \gota_N(u,v)-\gota_R(u,v) =-C_{n,s} \int_{\Omc} \Big( \frac{\beta^2(y)}{C_{n,s}^2\rho(y)} +\frac{\beta(y)}{C_{n,s}}\Big) u_R(y) v_R(y) \;dy 
         \leq 0,
    \end{align*}
  where we have used that $u_R,v_R\ge 0$ a.e. in $\RR^n$ by Remark \ref{rem-uR} (since \(u,v \geq 0\) a.e. in $\Omega$). Thus, \(\gota_N(u,v)\leq \gota_R(u,v)\). It follows from Theorem \ref{thm-29}  that \eqref{dom-2-2} holds.\\

{\bf Step 3}: Finally, \eqref{dom2} follows from \eqref{dom2-1} and \eqref{dom-2-2}. The proof is finished.
\end{proof}

\begin{remark}\label{rem-44}
{\em
    Assume that $\Omega$ is of class $C^1$. For an arbitrary regular Borel measure \(\mu\) on $\Omc$, one would like to define the following form:
    \begin{align}\label{F-mu}
        \gota_\mu(u,v):=\gota_N(u,v) + \int_{\Omc} u_N(y) v_N(y)d\mu
    \end{align}
    with
    \begin{align*}
    D(\gota_\mu):=\left\{u\in D(\gota_N)\cap C(\bOm):\; \int_{\Omc}|u_N(y)|^2\;d\mu<\infty\right\}.
    \end{align*}
By Lemma \ref{cont-ext-N} \(u_N \in C(\RR^n)\), if \(u \in C(\bOm)\) and \(\Om\) is of class \(C^1\).
Let us assume that the form $\gota_\mu$ is closable in $L^2(\Om)$ and denote its closure again by \(\gota_\mu\). Let \(T_\mu\) be the associated semigroup. 
We have the following situation.
\begin{enumerate}
\item It is clear that \(0\le T_\mu \leq T_N\). But the domination \(T_D \leq T_\mu\) is not true in general. Indeed, let $u\in D(\gota_D)$. Calculating we get that
    \begin{align*}
        \gota_D(u,u)-\gota_\mu(u,u) =C_{n,s} \int_{\Omc} \rho(y) u_N^2(y)\;dy - \int_{\Omc} u_N^2(y) d\mu.
    \end{align*}
    Hence, the domination \(T_D \leq T_\mu\)  holds if and only if 
    \begin{align}\label{DR}
        \int_{\Omc} u_N^2(y) d\mu \leq C_{n,s}\int_{\Omc} \rho(y) u_N^2(y)\;dy
    \end{align}
    for every  $u\in D(\gota_D)$.
    The estimate \eqref{DR} fails, for example, if one takes \(d\mu= 2 C_{n,s}\rho(y) dy\). 
    
\item    On the other hand, we have that for every  $u,v\in D(\gota_R)$ (by using  \eqref{N-to-R} and the equality in \eqref{mjw}),
    \begin{align*}
        \gota_R(u,v)=\gota_N(u,v)+ C_{n,s}\int_{\Omc} u_N(y) v_N(y) \frac{\beta(y) \rho(y)}{C_{n,s}\rho(y) + \beta(y)}\;dy 
    \end{align*}
    and 
    \begin{align*}
        \rho(y) - \frac{\beta(y) \rho(y)}{C_{n,s}\rho(y) + \beta(y)} = \frac{C_{n,s}\rho^2(y)}{C_{n,s}\rho(y) + \beta(y)} \geq 0.
    \end{align*}
   In that case, taking the measure $\mu$ as follows:   
    \begin{align*}
        d\mu= \frac{C_{n,s}\beta(y) \rho(y)}{C_{n,s}\rho(y) + \beta(y)} dy,
    \end{align*}
     we get that \(\gota_R=\gota_\mu\).
  \end{enumerate}
    }
    \end{remark}

Next, we show some contractivity properties of the three semigroups.

\begin{thm}\label{thm-dir-form}
The semigroups $T_D$, $T_R$ and \(T_N\) are submarkovian.
\end{thm}

\begin{proof}
Since $T_D$, $T_R$ and \(T_N\) are positivity-preserving, it suffices to show that they are \(L^\infty\)-contractive.
We prove the theorem in two steps.\\

{\bf Step 1}: We claim that  \(T_N\) is \(L^\infty\)-contractive.
By  \cite[Lemma 2.7]{War} we know  that 
\begin{align}\label{mw2}
    \mathcal{E}(f \wedge 1,f\wedge 1) \leq \mathcal{E}(f,f)
\end{align}
for every \(0\le f \in W^{s,2}_\Omega\).

Next,  let \(0\le u \in D(\gota_N)\). Then, \(u_N \in W^{s,2}_\Omega\) and by \eqref{mw2} we have 
\begin{align}\label{M3}
    \mathcal{E}(u_N \wedge 1,u_N \wedge 1) \leq \mathcal{E}(u_N,u_N)=\gota_N(u,u).
\end{align}
We want to show that
\begin{align*}
    \gota_N(u \wedge 1,u \wedge 1) = \mathcal{E}((u \wedge 1)_N,(u\wedge 1)_N) \leq \mathcal{E}(u_N \wedge 1,u_N\wedge 1).
\end{align*}
Observe that \(u_N \wedge 1 = (u \wedge 1)_N\) a.e. in \(\Omega\) (but not in $\Omb$). Calculating we get that
\begin{align}\label{M4}
    &\mathcal{E}(u_N \wedge 1,u_N\wedge 1)-\mathcal{E}((u \wedge 1)_N,(u \wedge 1)_N) \notag\\
    &= 
  C_{n,s}  \int_\Omega \int_{\Omc} \frac{\Big((u\wedge 1) (x)-(u_N \wedge 1)(y)\Big)^2}{\vert x-y\vert ^{n+2s}} dy \; dx\notag \\
    &\quad - C_{n,s}\int_\Omega \int_{\Omc} \frac{\Big((u \wedge 1)(x) -(u\wedge 1)_N(y)\Big)^2}{\vert x-y\vert ^{n+2s}} dy \; dx \notag\\
    &=  C_{n,s}\int_\Omega \int_{\Omc} \frac{ (u\wedge 1)^2 (x) -2 (u\wedge 1) (x) (u_N\wedge 1)(y) + (u_N \wedge 1)^2(y)}{\vert x-y\vert ^{n+2s}}\;dy\;dx\notag \\
    &\quad - C_{n,s}\int_{\Om}\int_{\Omc} \frac{ (u \wedge 1)^2(x) -2 (u \wedge 1)(x) (u\wedge 1)_N(y) + (u\wedge 1)_N^2(y)}{\vert x-y\vert ^{n+2s}} \;dy \; dx\notag \\
    &=  C_{n,s}\int_{\Omc} \Big( -2 \rho(y) (u\wedge 1)_N^2 (y) +  \rho(y) (u_N \wedge 1)^2(y)\Big)\;dy    \notag \\ 
    &\quad + C_{n,s}\int_{\Omc}\Big(2 \rho(y) (u \wedge 1)_N^2(y)  - \rho(y) (u\wedge 1)_N^2(y)\Big)\; dy \notag \\
    &=  C_{n,s}\int_{\Omc}\Big(  \rho(y) (u_N \wedge 1)^2(y)  -2 \rho(y) (u\wedge 1)_N^2(y) +\rho(y) (u\wedge 1)_N^2(y)\Big)\; dy\notag  \\
    &=  C_{n,s}\int_{\Omc} \rho(y)\Big(  (u_N \wedge 1)^2(y) -  (u\wedge 1)_N^2 (y)  \Big)\;dy 
    \geq 0,
\end{align}
where we have used that $(u_N\wedge 1)^2\ge (u\wedge 1)_N^2$ a.e. in $\Omb$.
It follows from \eqref{M3} and \eqref{M4} that
\begin{align}\label{M5}
    \gota_N(u \wedge 1,u\wedge 1) \leq \gota_N(u,u).
\end{align}
By Remark \ref{rem1} the inequality \eqref{M5} is equivalent to the \(L^\infty\)-contractivity of \(T_N\).\\

{\bf Step 2}: Since  \(T_N\) is \(L^\infty\)-contractive (by Step 1), it follows from the domination \eqref{dom2} that $T_D$ and $T_R$ are also   \(L^\infty\)-contractive. The proof is finished.
\end{proof}

Next, we have the following ultracontractivity result.

\begin{thm}
The following assertions hold.
\begin{enumerate}
\item There is a constant $C>0$ such that 
\begin{align}\label{s1}
\max\left\{\|T_D(t)\|_{\mathcal L(L^1(\Om),L^\infty(\Om))}, \|T_R(t)\|_{\mathcal L(L^1(\Om),L^\infty(\Om))}\right\}\le Ct^{-\frac{n}{2s}},\;\;\;\forall\; t>0.
\end{align}

\item There is a constant $C>0$ such that 
\begin{align}\label{s2}
\|T_N(t)\|_{\mathcal L(L^1(\Om),L^\infty(\Om))}\le Ce^tt^{-\frac{n}{2s}},\;\;\;\forall\; t>0.
\end{align}
\end{enumerate}
\end{thm}

\begin{proof}
Recall that $D(\gota_D)=\widetilde W_0^{s,2}(\Omega)$. It follows from \eqref{def-formaN} and \eqref{def-formaR} that the continuous embeddings $D(\gota_N)$, $D(\gota_R)\hookrightarrow W^{s,2}(\Om)$ hold. In addition, by Theorem \ref{thm-dir-form} we have that $(\gota_D,D(\gota_D))$, $(\gota_R,D(\gota_R))$ and $(\gota_N,D(\gota_N))$ are Dirichlet forms on $L^2(\Om)$.
Hence, using Remark \ref{rem-sob} we can deduce that  $(\gota_D,D(\gota_D))$, $(\gota_R,D(\gota_R))$ and $(\gota_N,D(\gota_N))$ satisfy all the hypotheses in Theorem \ref{thm-main} with $q_{\gota}=\frac{n}{n-2s}>1$. It also follows from Remark \ref{rem-sob} that the embeddings $D(\gota_D)$, $D(\gota_N)$, $D(\gota_R)\hookrightarrow L^2(\Om)$ are compact.
We have shown that the operators $A_D$, $A_R$,  $A_N$, and the semigroups $T_D$, $T_R$, $T_N$ satisfy all the assertions in Theorem \ref{thm-main} and Remark \ref{rem-210}.  Thus, the estimate \eqref{s1} follows from Theorem \ref{thm-main} and Remark \ref{rem-210} together with the fact that the first eigenvalues of $A_D$ and $A_R$ are strictly positive. The estimate \eqref{s2} also follows from Theorem \ref{thm-main},  Remark \ref{rem-210} and the fact that for $A_N$, its first eigenvalue is zero, as the constant function $1\in D(\gota_N)$ and $\gota_N(1,1)=0$. The proof is finished.
\end{proof}

We conclude the paper by giving some open problems.

\section{Open problems}\label{open}

In this section we give some interesting open problems related to the three Dirichlet forms and semigroups investigated in the previous sections.

\begin{enumerate}

\item[(1)] {\bf Inner regularity of solutions to the fractional Neumann and Robin problems.} In the case of the fractional Dirichlet problem, it has been shown in \cite{BWZ1} that weak solutions of \eqref{DiPr} belong to $W_{\rm loc}^{2s,2}(\Omega)$.

{\it As we have already mentioned in Section \ref{sec-3}, it is then a natural question to ask if the same inner regularity result holds for the Neumann and Robin problems}. 

This very interesting problem will be a topic of a future investigation.\\

\item[(2)] {\bf From weak to strong solutions.} {\it As we have mentioned in Section \ref{sec-3} we do not know if weak solutions of the Dirichlet, Neumann and Robin problems are strong solutions}. 

This is an interesting subject and deserves to be clarified. This will also be a topic of a future investigation.\\

\item[(3)] {\bf Kernel estimates for the semigroups $T_R$ and $T_N$.} It follows from Theorem  \ref{thm-main} that each of the semigroups $T_D$, $T_R$ and $T_N$ is given by a kernel $K:(0,\infty)\times\Omega\times\Omega\to \RR$ and $K(t,\cdot,\cdot)$ belongs to $L^\infty(\Om\times\Om)$ for every $t>0$. Let us denote by $K_D$, $K_R$ and $K_N$ the kernels of $T_D$, $T_R$ and $T_N$, respectively.
It has been shown in \cite{Blu,Chen2}  that 
there are two constants $0<C_1\le C_2$ such that for a.e $x,y\in\Omega$ and $t>0$, 
\begin{equation}\label{KER}
C_1t^{-\frac{N}{2s}}\left(1+|x-y|t^{-\frac{1}{2s}}\right)^{-(N+2s)}\le K_D(t,x,y)\le C_2t^{-\frac{N}{2s}}\left(1+|x-y|t^{-\frac{1}{2s}}\right)^{-(N+2s)}.
\end{equation}

{\it What are the corresponding estimates for the kernels $K_R(t,\cdot,\cdot)$ and $K_N(t,\cdot,\cdot)$?}\\

\item[(4)] {\bf Analyticity on $L^1$ of the semigroups $T_R$ and $T_N$.} By Theorem \ref{thm-main} again the semigroups $T_D$, $T_R$ and $T_N$ are analytic on $L^p(\Omega)$ for every $1<p<\infty$. Very recently, using \eqref{KER}, it has been shown in \cite{KSW} that the semigroup $T_D$ is also analytic of angle $\frac{\pi}{2}$ on $L^1(\Om)$.

 {\it The analyticity on $L^1(\Om)$ of the semigroups $T_R$ and $T_N$ remains an open problem.}\\

\item[(5)] {\bf Sandwiched semigroups.} Assume for simplicity that $\Omega$ has a Lipschitz continuous boundary. In the case of the Laplace operator, the relative capacity $\operatorname{Cap}_{\bOm}$ has been defined in  \cite{ArWa1,ArWa2,War-T} for an arbitrary set $E\subset\bOm$ by
\begin{align*}
\operatorname{Cap}_{\bOm}(E):=\inf\Big\{\|u\|_{W^{1,2}(\Om)}^2:\; u\in W^{1,2}(\Omega):\;\exists\; O\subset\RR^n\;\mbox{ open such that } \\
\hfill E\subset O\mbox{ and } u\ge 1 \mbox{ a.e. in }\; \Omega\cap O\Big\}.
\end{align*}
 Let  $\eta$ be a regular Borel measure on $\pOm$. Assume that $\eta$ is absolutely continuous with respect to $\operatorname{Cap}_{\bOm}$ in the sense that
 \begin{align}\label{abs}
 \operatorname{Cap}_{\bOm}(B)=0\;\Longrightarrow\;\eta(B)=0\;\mbox{ for any Borel set } B\subset\pOm.
 \end{align}
  Let
\begin{align*}
D(\gota^\eta):=\Big\{u\in W^{1,2}(\Omega):\;\int_{\pOm}|\tilde u|^2\;d\eta<\infty\Big\},
\end{align*}
where $\tilde u$ denotes the relative quasi-continuous version of $u$, and define the closed bilinear form $\gota^\eta: D(\gota^\eta)\times D(\gota^\eta)\to\RR$ in $L^2(\Om)$ by
\begin{align*}
\gota^\eta(u,v):=\int_{\Omega}\nabla u\cdot\nabla v\;dx+\int_{\pOm}\tilde u\tilde v\;d\eta.
\end{align*}
It has been shown in \cite{ArWa1,War-T} that the semigroup $T^\eta$ associated with $(\gota^\eta,D(\gota^\eta))$ satisfies
\begin{align*}
0\le T^D\le T^\eta\le T^N,
\end{align*}
in the sense of \eqref{dom},
where $T^D$ is the semigroup on $L^2(\Omega)$ associated with the form
\begin{align*}
\gota^D(u,v):=\int_{\Omega}\nabla u\cdot\nabla v\;dx, \;\;u,v\in D(\gota^D):=W_0^{1,2}(\Omega),
\end{align*}
and $T^N$ is the semigroup on $L^2(\Om)$ associated with $(\gota^N,D(\gota^N))$ (see \eqref{lnbc}). Conversely, they have also shown that any symmetric  semigroup on $L^2(\Om)$ associated with a regular and local Dirichlet form (see e.g. \cite[Chapter 1]{Fuk} for the definition of a regular and local form) and sandwiched between $T^D$ and $T^N$, is always given by $T^\eta$ for some regular Borel measure $\eta$ on $\pOm$ satisfying \eqref{abs}.

In the case of the fractional Laplace operator, we have seen in Remark \ref{rem-44}(a) that for the form $\gota_\mu$ (recall that here $\mu$ is a regular Borel measure on $\Omc$) given in \eqref{F-mu}, the associated semigroup $T_\mu$ is not always sandwiched between $T_D$ and $T_N$. Of course in this case we have shown that $T_\mu\le T_N$, but the domination $T_D\le T_\mu$ is not always true. 

In addition, consider the form $\gota: D(\gota_D)\times D(\gota_D)\to\RR$ in $L^2(\Omega)$ given by 
\begin{align*}
\gota(u,v):=\frac 12\Big(\gota_D(u_D,v_D)+\gota_N(u_N,v_N)\Big),
\end{align*}
where we recall that $D(\gota_D)=\widetilde W_0^{s,2}(\Omega)$, and $u_D$, $u_N$ are given in \eqref{uD} and \eqref{ext-N}, respectively.
It is easy to see that $\gota$ is symmetric, closed and densely defined. Let $A$ be the selfadjoint operator on $L^2(\Omega)$ associated with $(\gota,D(\gota))$ and $T$ the associated semigroup. Then, we can easily show that the domination $0\le T_D\le T\le T_N$ holds in the sense of \eqref{dom}. But  it is not clear if $A$ is a realization in $L^2(\Omega)$ of $(-\Delta)^s$. Therefore, a natural question arises.
\begin{center}
{\it Let $(\gota_\mu, D(\gota_\mu))$ be the form defined in Remark \ref{rem-44}.
Let $T$ be a symmetric semigroup on $L^2(\Om)$ satisfying $0\le T_D\le T\le T_N$  in the sense of \eqref{dom} and let $(\gota,D(\gota))$ be the closed bilinear form on $L^2(\Om)$ associated with $T$. Under which conditions on $(\gota,D(\gota))$ does a measure \(\mu\) exist such that $\gota=\gota_\mu$?}
\end{center}
\end{enumerate}

\subsection*{\bf Acknowledgement} The authors would like to thank the referee for the careful reading of the manuscript and the valuable suggestions, which have been very helpful in improving the paper.

\bibliographystyle{plain}
\bibliography{refs}

\end{document}